\documentclass[twoside,leqno]{article}

\usepackage{siamproceedings}

\usepackage[T1]{fontenc}
\usepackage{amsfonts}
\usepackage{graphicx}
\usepackage{epstopdf}
\usepackage{enumitem}
\usepackage{algorithmic}
\ifpdf
  \DeclareGraphicsExtensions{.eps,.pdf,.png,.jpg}
\else
  \DeclareGraphicsExtensions{.eps}
\fi

\newsiamremark{remark}{Remark}
\newsiamremark{hypothesis}{Hypothesis}
\crefname{hypothesis}{Hypothesis}{Hypotheses}
\newsiamthm{claim}{Claim}
\newsiamthm{conjecture}{Conjecture}
\newsiamremark{question}{Question}

\usepackage{amsopn}
\DeclareMathOperator{\diag}{diag}

\begin{document}

\newcommand\relatedversion{}
\renewcommand\relatedversion{\thanks{The full version of the paper can be accessed at \protect\url{https://arxiv.org/abs/0000.00000}}}

\newcommand{\bi}{\begin{itemize}}
\newcommand{\ei}{\end{itemize}}
\newcommand{\op}{\operatorname}
\newcommand{\Out}{\op{Out}}
\newcommand{\Aut}{\op{Aut}}
\newcommand{\Inn}{\op{Inn}}
\newcommand{\MCG}{\op{MCG}}
\newcommand{\Om}{\Omega}
\newcommand{\PS}{\op{PS}}
\newcommand{\bb}{\mathbb}
\newcommand{\tsubset}{<}
\newcommand{\tsubseteq}{<}
\newcommand{\be}{\begin{equation}}
\newcommand{\ee}{\end{equation}}
\newcommand{\Ga}{\Gamma}
\newcommand{\bc}{\mathbb C}
\newcommand{\R}{\mathbb R}
\renewcommand{\H}{\mathbb H}
\newcommand{\Z}{\mathbb Z}
\newcommand{\N}{\mathbb N}
\newcommand{\ga}{\gamma}
\newcommand{\ka}{\kappa}
\newcommand{\la}{\lambda}
\newcommand{\La}{\Lambda}
\newcommand{\inte}{\op{int}}
\newcommand{\ba}{\backslash}
\newcommand{\ov}{\overline}
\newcommand{\cl}{\overline}
\newcommand{\til}{\tilde}
\renewcommand{\cal}{\mathcal}
\newcommand{\br}{\mathbb R}
\newcommand{\SO}{\op{SO}}
\newcommand{\Isom}{\op{Isom}}
\newcommand{\PSL}{\op{PSL}}
\newcommand{\F}{\cal F}
\newcommand{\hu}{\op{hull}}
\newcommand{\core}{\op{core}}
\newcommand{\coree}{\op{core}_\e(M)}
\newcommand{\bH}{\mathbb H}
\newcommand{\mT}{\mathsf T}
\newcommand{\BFM}{\op{BF} M}

\newcommand{\RF}{\op{RF}}
\newcommand{\FM}{\op{F}M}
\newcommand{\hull}{\op{hull}}
\newcommand{\stab}{\op{Stab}}
\newcommand{\Stab}{\op{Stab}}
\newcommand{\diam}{\op{diam}}
\newcommand{\vol}{\op{vol}}
\newcommand{\q}{\mathbb Q}
\newcommand{\G}{\Gamma}
\newcommand{\m}{\mathsf{m}^{\BMS}}
\renewcommand{\O}{\mathcal O}
\newcommand{\mS}{\mathscr{S}}
\newcommand{\mG}{\mathscr{G}}\newcommand{\mH}{\mathscr{H}}
\newcommand{\T}{\op{T}}
\renewcommand{\frak}{\mathfrak}
\renewcommand{\v}{\mathsf v}
\newcommand{\e}{\varepsilon}
\newcommand{\BR}{\op{BR}}
\newcommand{\tdg}{\partial_\infty^{(2)}\Ga}
\newcommand{\BMS}{\op{BMS}}
\renewcommand{\L}{\mathcal L}
\newcommand{\fa}{\mathfrak a}
\newcommand{\pc}{P^{\circ}}
\renewcommand{\i}{\op{i}}
\newcommand{\M}{\mathcal{M}}
\renewcommand{\S}{\mathbb S}
\newcommand{\hd}{\op{dim}{}\Lambda}
\newcommand{\so}{\SO^\circ}
\newcommand{\ns}{\not\sim}
\newcommand{\s}{\sigma}
\newcommand{\cT}{\cal T_{\La_\rho}}
\newcommand{\Gr}{\Gamma_\rho}
\newcommand{\C}{\cal C}
\newcommand{\Tei}{\op{Teich}}
\newcommand{\id}{\op{id}}
\newcommand{\dGa}{\mathfrak R_{\op{disc}}(\Ga)}
\newcommand{\Mob}{\text{M\"ob}}

\newcommand{\PGL}{\op{PGL}}
\newcommand{\GL}{\op{GL}}
\newcommand{\E}{\cal E}
\newcommand{\Fr}{\op{F}}
\renewcommand{\P}{\mathbb{P}}
\newcommand{\ess}{\mathsf{E}}
\newcommand{\B}{\mathcal{B}}
\newcommand{\fg}{\frak g}

\newcommand{\lox}{\op{lox}}
\newcommand{\p}{\mathsf{p}}

\newcommand{\Fund}{\op{Fund}}
\newcommand{\fk}{\mathfrak{k}}
\newcommand{\fp}{\mathfrak{p}}
\newcommand{\rank}{\op{rank}}
\newcommand{\Lie}{\op{Lie}}
\newcommand{\U}{\mathcal{U}}
\newcommand{\Leb}{\op{Leb}}
\newcommand{\supp}{\op{supp}}
\newcommand{\sa}{\mathsf A}
\renewcommand{\epsilon}{\e}
\renewcommand{\t}{\theta}
\renewcommand{\d}{\mathsf{d}}
\newcommand{\sq}{\mathsf{Q}}
\newcommand{\SL}{\op{SL}}
\newcommand{\codim}{\op{codim}}
\def\g{\gamma}
\newcommand{\TG}{\mathscr T_\Ga}
\renewcommand{\c}{\mathbb C}

\newcommand{\vv}{\mathtt{v}}
\newcommand{\cM}{\mathcal M}
\renewcommand{\F}{\op{F}}
\renewcommand{\diag}{\op{diag}}
\newcommand{\bS}{\mathbb S}
\renewcommand{\P}{\cal P}
\newcommand{\hc}{\widehat{\mathbb C}}
\newcommand{\Vol}{\vol}
\newcommand{\da}{\delta_{\mathsf A}}
\newcommand{\Inv}{\op{Inv}} 
\newcommand{\RFM}{\op{RF}\cM}
\newcommand{\rdi}{\mathfrak R_{\op{disc}} (\Gamma)}
\renewcommand{\u}{\mathsf u}
\newcommand{\w}{\mathsf w}
\newcommand{\Haar}{\op{Haar}}
\newcommand{\z}{\mathbb Z}
\title{Dynamics and Rigidity through the Lens of Circles}

\author{Hee Oh\thanks{Department of Mathematics, Yale University, New Haven, CT 06511 (\email{hee.oh@yale.edu})}}

\maketitle

\pagenumbering{arabic}
\setcounter{page}{1}

\begin{abstract} We report on recent developments in the dynamics and rigidity of infinite-volume homogeneous spaces through the lens of circles. By addressing four natural questions about circle packings, we highlight the interplay between dynamics, geometry, and rigidity that defines the emerging frontier of homogeneous dynamics.\end{abstract}

\section{Introduction}
This article explores the interplay between dynamics, geometry, and rigidity in infinite-volume homogeneous spaces, viewed through the lens of circles. Motivated by four natural questions on circle counting, orbit closures, rigidity, and torus counting, we present results that highlight two complementary aspects of homogeneous dynamics. Orbit-closure classification and rigidity phenomena reflect the underlying geometric and algebraic structure, while counting and equidistribution are governed by mixing in infinite volume. Together, these circle problems show how dynamics and rigidity in infinite volume arise naturally and illuminate the broader frontier of homogeneous dynamics beyond the finite-volume setting.

\subsection{Circle packings} 
Circles have fascinated mathematicians since the time of the ancient Greeks.
By a circle packing $\cal P$, we mean a configuration of countably many circles in the complex plane $\mathbb C$. 
How can such configurations be constructed? The earliest systematic example is due to Apollonius of Perga (262-190 BC), who proved that  three mutually tangent circles in the plane admit precisely two circles tangent to all three. Starting with four mutually tangent circles and repeatedly inserting new circles tangent to three of the existing ones, as prescribed by Apollonius's theorem, one obtains the celebrated Apollonian circle packing.  Figure \ref{ap} illustrates the first few generations, with circles labeled by the reciprocals of their radii.
\begin{figure}[htbp]\centering
  \includegraphics [height=2.7cm]{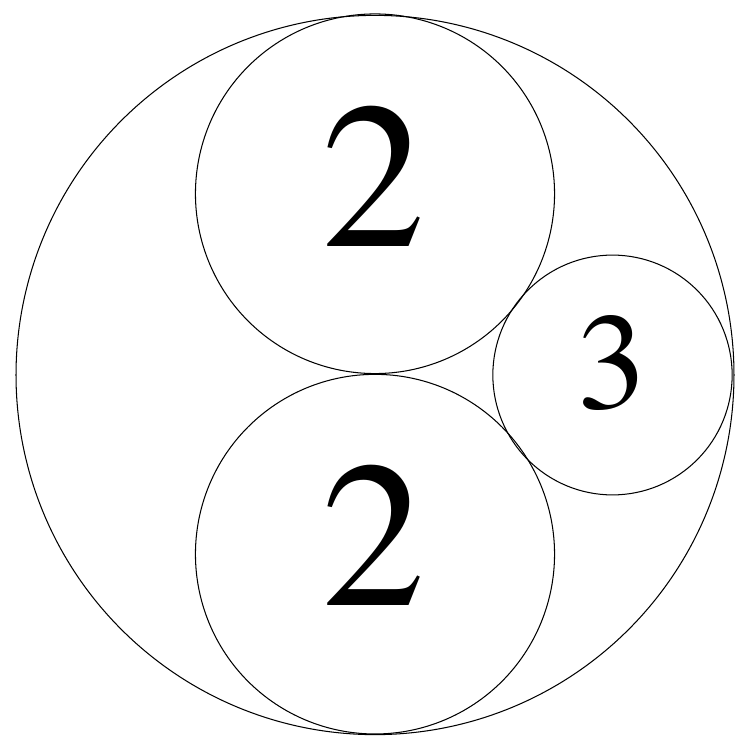}
  \includegraphics [height=2.7cm]{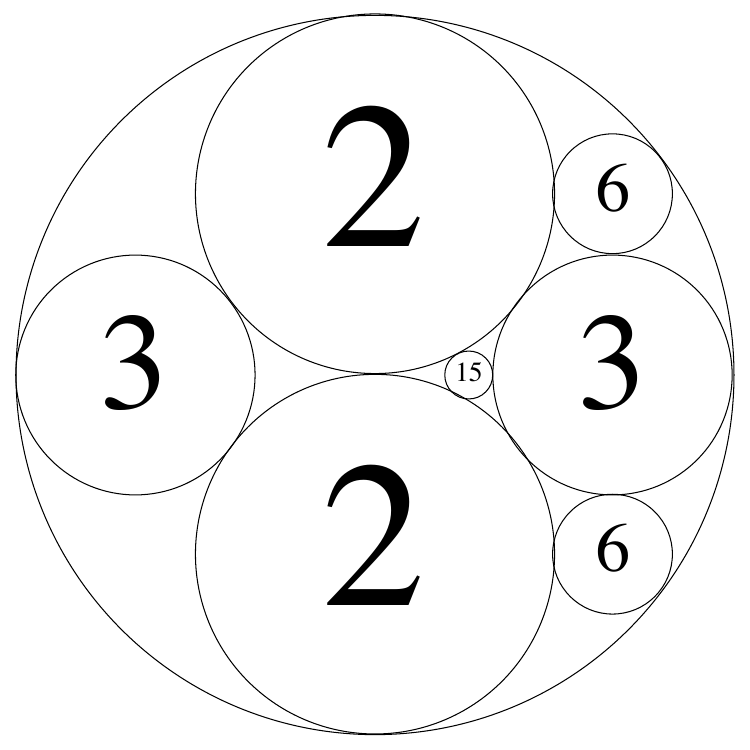}
  \includegraphics [height=2.7cm]{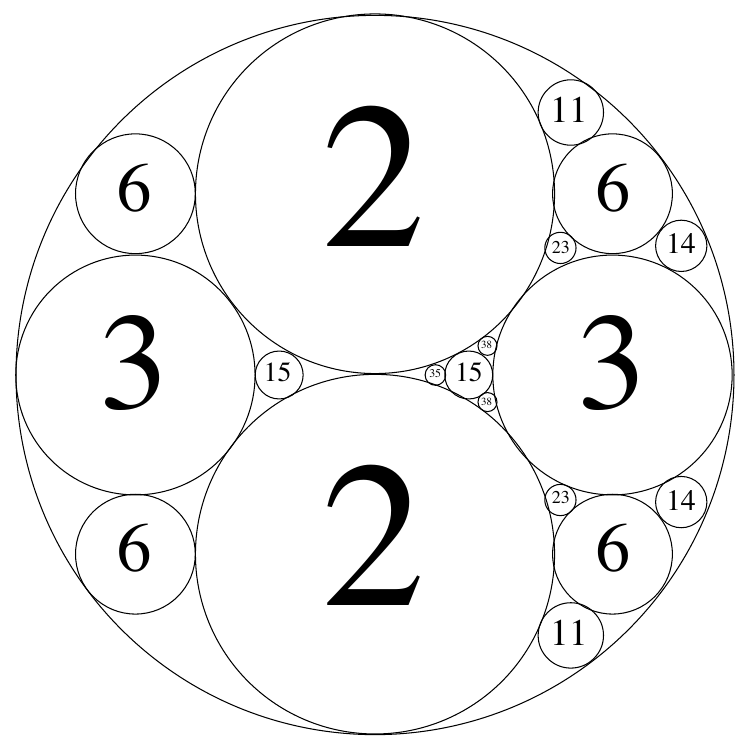}
   \includegraphics [height=2.7cm]{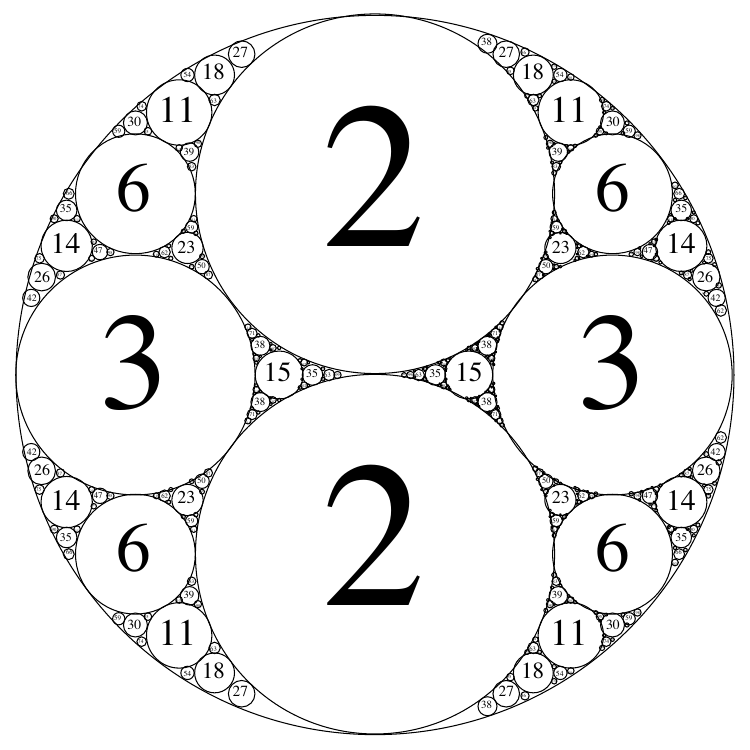} 
   \caption{First few generations of Apollonian circle packings} \label{ap} 
\end{figure}

Another striking example
is a Sierpi\'nski-type circle packing, shown in Figure \ref{S1}. 
\begin{figure}[htbp] \begin{center}
  \includegraphics [height=3cm]{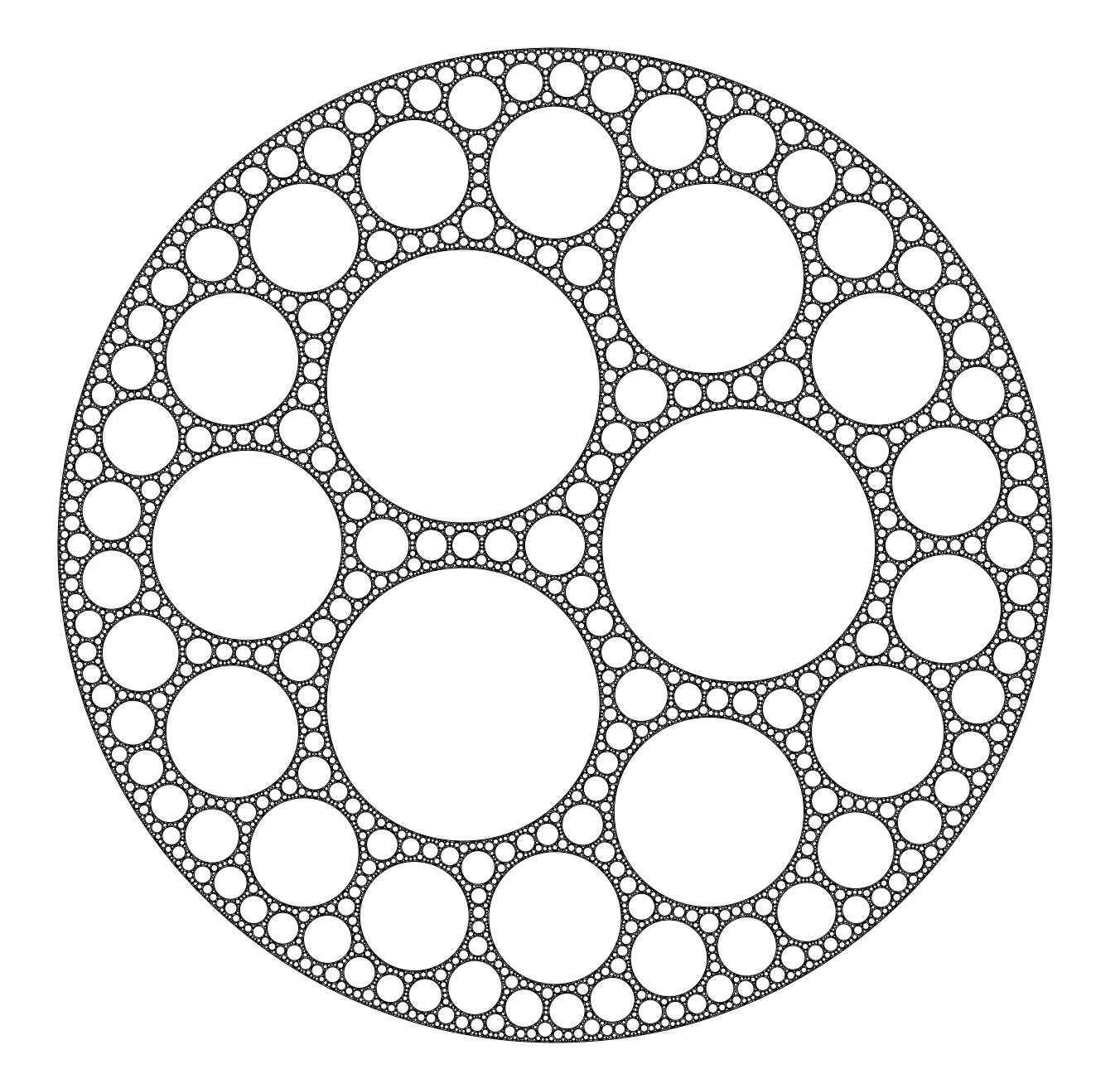} 
\end{center}
\caption{A Sierpi\'nski-type circle packing}\label{S1}
\end{figure}

Where do such intricate circle packings arise? A unifying perspective is that
both Apollonian and Sierpi\'nski-type  packings emerge as $$\text{limit sets of Kleinian groups. }$$

 A Kleinian group is
a discrete subgroup of $\PSL_2(\c)$. Such groups act on the Riemann sphere  $\hc=\c\cup\{\infty\}$ by M\"obius transformations. Explicitly,  for any $a,b,c,d\in \c$ with $ad-bc=1$ and $z\in \hc$,
    \be\label{mo} \begin{pmatrix} a& b \\ c& d\end{pmatrix} \cdot z=\frac{az+b}{cz+d}. \ee 
The limit set $\La\subset \hc $ of a Kleinian group $\Ga$ is the collection of accumulation points of an orbit $\Ga (z)$ for any $z\in \hc$.

\subsection{Four motivating questions} We regard a circle packing $\cal P$ both as the set of circles and as their union in $\c$; thus writing
$C\in \cal P$ refers to one of the circles, while $\overline {\cal P}$ denotes its closure in $\hc$.
Given a circle packing $\cal P$ with symmetry group $\Ga<\PSL_2(\c)$, meaning
 $\Ga \cal P=\cal P$ and $\overline \P=\La$, the following motivating questions are natural:
 \begin{enumerate}
  \item {\it Circle counting:} How many circles in $\mathcal P$ have radii at least $t$ as $t \to 0$?  

  \item {\it Orbit closures:} For a given circle $C \subset \widehat{\mathbb C}$, not necessarily in $\mathcal P$, what is the closure of $\Gamma C$ in the space of circles?  

  \item {\it Rigidity:} How can we decide whether a $\Gamma$-equivariant embedding 
  $f:\overline{\mathcal P} \to \hc$ must be a M\"obius transformation?  

  \item {\it Torus counting:} For $f$ as in question 3, how many tori $(C, f(C)) \in \mathcal P \times f(\mathcal P)$ have volume at least $t$ as $t \to 0$ (see Figure~\ref{torus} for an illustration)?
\end{enumerate}

\begin{figure}[htbp] \label{torus} \begin{center}
  \includegraphics [height=3.5cm]  {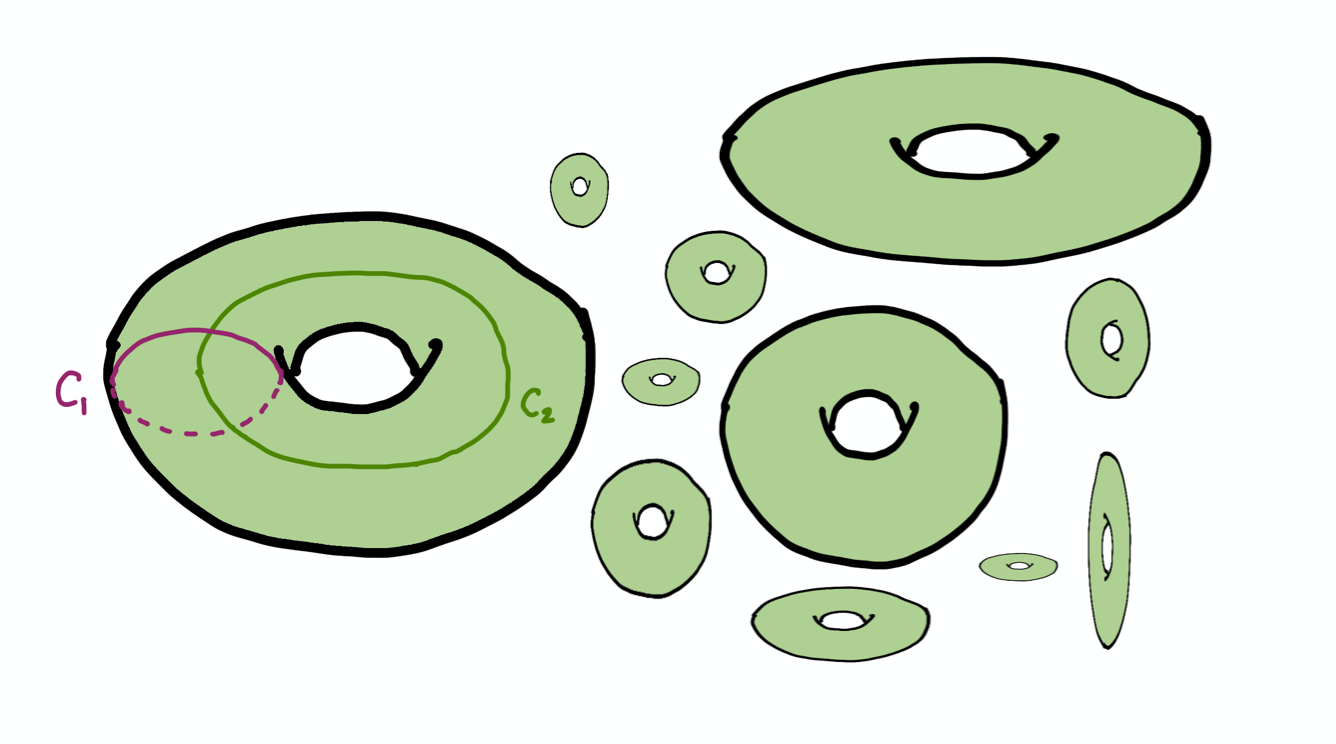} 
\end{center}\caption{Torus packing}
\end{figure}

While simple to state, these problems lead directly into subtle questions about dynamics on infinite-volume homogeneous spaces, ranging from rank one to higher rank. Infinite-volume marks precisely where classical finite-volume methods cease to apply and new dynamical phenomena emerge. 

\begin{itemize}
  \item  Questions 1 and 2 concern the distribution of geodesic planes 
in hyperbolic $3$-manifolds and involve, respectively, the dynamics of frame flows and unipotent flows on infinite-volume hyperbolic $3$-manifolds. 
 In sharp contrast to the finite-volume case, the ambient geometry and topology exert decisive influence on the answers. 
\item Question 3 addresses representation rigidity of  Kleinian groups. Our approach highlights new perspectives on rigidity phenomena of Kleinian groups by interpreting them through
dynamics on higher rank homogeneous spaces.
\item Question 4 is a higher-rank analogue of question 1, relying on the dynamics of diagonal flows in higher-rank spaces of infinite volume. 
\end{itemize}
Together, these problems illustrate the intricate interplay between geometry, dynamics, and rigidity that underpins much of this article. While the finite-volume case is by now well understood, their resolution in the infinite-volume setting reveals
 genuinely new and subtler dynamical phenomena. In this perspective, they may be viewed as infinite-volume analogues of several landmark results:  
 Duke, Rudnick, and Sarnak \cite{DRS} and Eskin and McMullen \cite{EM} on counting and equidistribution (questions 1 and 4),
 Ratner and Shah (\cite{ratner_top}, \cite{Shah_1991}) on orbit-closure classification (question 2), and
 Mostow, Prasad, and Sullivan rigidity (\cite{Mostowbook}, \cite{Prasad1973strong}, \cite{Sullivan1981ergodic}) (question (3)).

\subsection{Organization} The paper is organized in two parts. The first part presents results addressing the four motivating questions about circles: circle-counting, orbit closures, representation rigidity, and torus-counting. The second part focuses on mixing phenomena in infinite volume: first in rank one, where we discuss local, exponential, and uniform exponential mixing with applications to the affine sieve, and then in higher rank, where we develop analogues for diagonal flows on self-joining quotients.
These mixing results form the main analytic tools underlying the counting and equidistribution problems.

Throughout, I have aimed to maintain an informal and somewhat conversational style.
Some statements are presented without full references or detailed historical context; these 
can be found in the cited papers.

\subsection*{Acknowledgement} I am deeply grateful to all my collaborators whose joint works are presented in this article, including
several of my former students. I owe special thanks to my Ph.D. advisor, Gregory Margulis, who introduced me to homogeneous dynamics; to Peter Sarnak, who opened my eyes to discrete subgroups of infinite covolume through Apollonian circle packings; and to Curtis McMullen, who led me  to Sierpi\'nski-type circle packings and the broader world of circle packings.  I also thank Dennis Sullivan for foundational ideas that have inspired my research over the past decade. 

 I am grateful to C. McMullen and Yongquan Zhang for permission to use Figures \eqref{S1}, \eqref{limitex}, \eqref{twol}, and \eqref{f:def}, and to Joy Kim for her help with the most of the other figures.

\section{Circle counting and Kleinian groups}\label{s:cc}
\subsection{Counting Apollonian circles} We begin by recalling the following circle-counting theorem proved in joint work with Kontorovich. We write $\op{rad} (C) $ for the Euclidean radius of a circle $C$.

\begin{theorem} [\cite{KO}]\label{ko}For any bounded Apollonian packing $\cal P$, there exists a constant $c_{\cal P}>0$ such that
$$\#\left\{C\in \cal P: \op{rad}( C) \ge \tfrac{1}t\right\}\sim c_{\cal P}\, t^{\da}\quad\text{as $t\to \infty$}$$
where $\delta_{\mathsf A}$ is the Hausdorff dimension of the closure $\overline{\cal P}$. 
\end{theorem}
The dimension $\delta_{\mathsf A}$ is independent of the choice of $\cal P$ (so $\mathsf A$ stands for "Apollonian") and has been estimated as $\da\approx 1.3057$. Circles in $\cal P$ are moreover equidistributed with respect to the $\delta_{\mathsf A}$-dimensional Hausdorff measure $\mathcal H^{\da}$ on $\overline{\cal P}$, defined using the Euclidean metric on $\mathbb C$, as established in joint work with Shah.
\begin{theorem}[\cite{OS_inv}]\label{os}There exists a constant $\mathsf{c_A}>0$ such that for any Apollonian packing $\cal P$, we have the following: for any region $R \subset \c$ bounded by a piecewise $C^1$-curve,
$$\#\left\{C\in \cal P: \op{rad} (C) \ge \tfrac{1}t,\; C\cap R\ne \emptyset  \right\} \sim {\mathsf c_{\mathsf A}}\, \mathcal H^{\da}(R\cap \overline{\cal P}) \, t^{\da}\quad\text{as $t\to \infty$.}$$ 
\end{theorem}

 The symmetry group $\G_{\cal P}$ of $\cal P$ is generated by the inversions with respect to the four dual circles (orthogonal to three of mutually tangent circles) corresponding to a quadruple of tangent circles in $\cal P$ (see the four red circles in Figure \ref{dual}). 
\begin{figure}[htbp]
 \begin{center}
    \includegraphics[height=3cm]{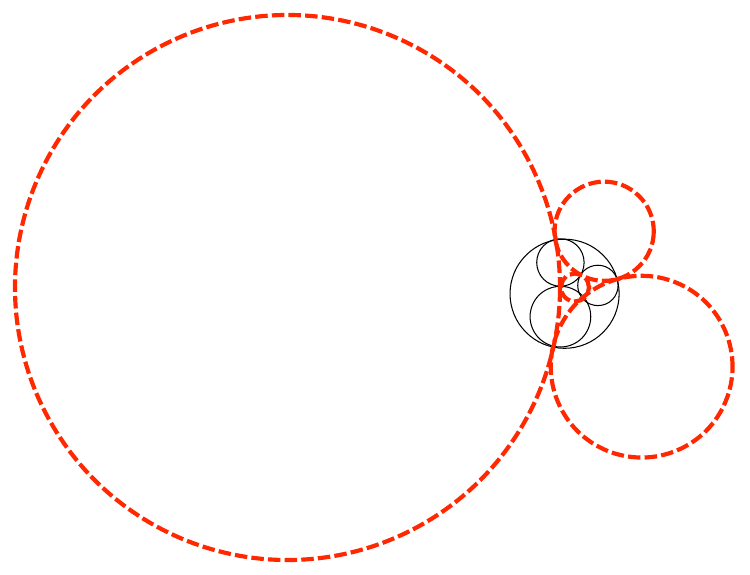}
 \end{center}\caption{Dual circles}\label{dual}
\end{figure} 

It is crucial in the above theorems that  $\G_{\cal P}$ is a {\it geometrically finite} Kleinian group whose limit set is equal to the closure $\overline{\cal P}$.

\subsection{Kleinian groups} Let us recall some background on Kleinian groups.  Such groups arise naturally as holonomy representations of fundamental groups of complete hyperbolic $3$-manifolds.
The hyperbolic $3$-space $\bH^3$ is the unique simply connected complete Riemannian manifold of dimension $3$ with constant sectional curvature $-1$. We use the upper half-space model:
$$\bH^3=\{(x_1, x_2, y): y>0\}\quad \text{with } 
ds=\frac{\sqrt{dx_1^2+dx_2^2+dy^2}}{y}.$$ 
The geometric boundary of $\bH^3$ is the Riemann sphere $\hc$, with 
the plane $(x_1, x_2, 0)$ identified with the complex plane $\c$. 
Geodesics (resp., geodesic planes) in $\bH^3$ are either vertical lines (resp.,  vertical planes) or  semicircles (resp., hemispheres) perpendicular to the plane $\c$.
\vspace{-0.5em}
\begin{figure}[htbp] \begin{center}    \includegraphics[height=4cm]{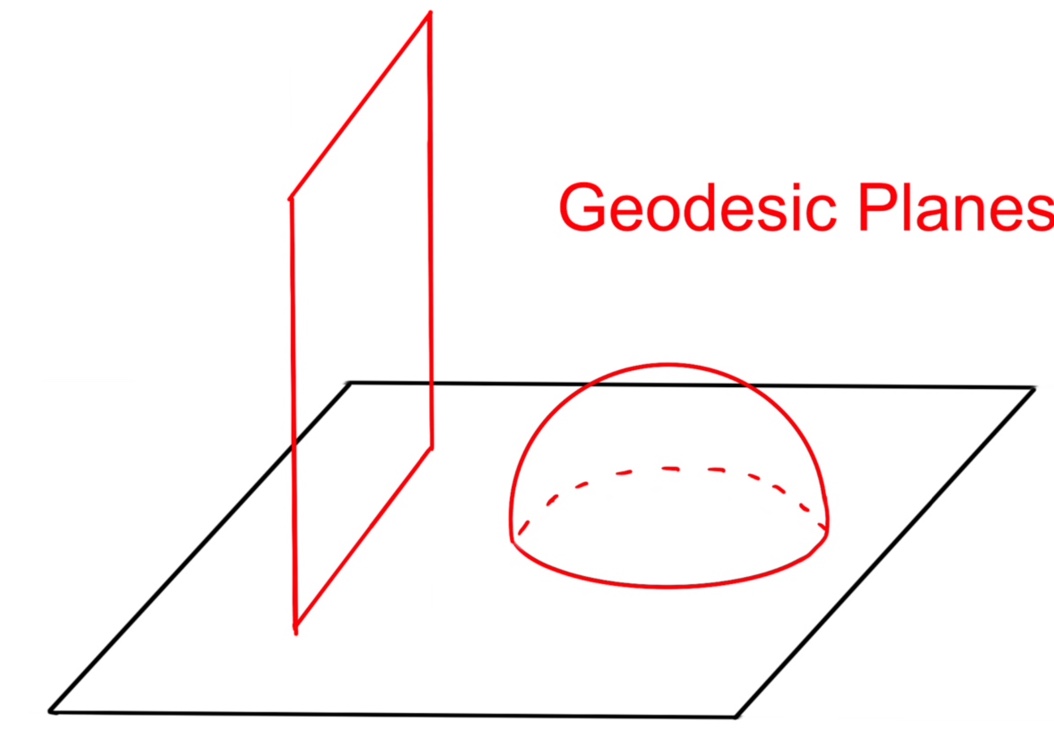} \end{center}\caption{Geodesic planes}\label{geop} \end{figure} 
The group $\Mob(\hc)$ of M\"obius transformations is generated by inversions with respect to circles in $\hc$, where an inversion in a circle sends each point to its reflection through the circle with respect to the Euclidean metric.
The group $\op{PSL_2}(\mathbb C)$ acts on $\hc$ by orientation preserving M\"obius transformations as in \eqref{mo}.
This action extends to an isometric action on $\bH^3$.
The Poincar\'e extension theorem identifies:
 $$\op{PSL}_2(\mathbb C)=\op{Isom}^+(\bH^3)\quad \text{and}\quad \Mob (\hc)=\op{Isom}(\bH^3).$$ 
 Let $\cal C$ denote the space of circles in $\hc$ equipped with its natural topology in which
 $C_i\to C$ if the Hausdorff distance between $C_i$ and $C$ tends to zero.
A classical characterization states that a homeomorphism $f$ of $\hc$ is a M\"obius transformation if and only if it preserves circles: $f(\cal C)=\cal C$.

\begin{definition} A discrete subgroup $\Gamma$ of $G=\PSL_2(\c)$ is called a Kleinian group.
\end{definition}
For general background on Kleinian groups, see \cite{Marden2016hyperbolic} and \cite{Matsuzaki1998hyperbolic}.
In this article, we assume  all Kleinian groups are torsion-free and non-elementary, that is, they do not have an abelian subgroup of finite index.
Each element of a Kleinian group is therefore  either loxodromic (conjugate to a diagonal element whose diagonal entries
have modulus not equal to $1$) or parabolic (conjugate to a strictly upper triangular matrix).

  \begin{figure}[htbp] \begin{center}
  \includegraphics [height=4cm]{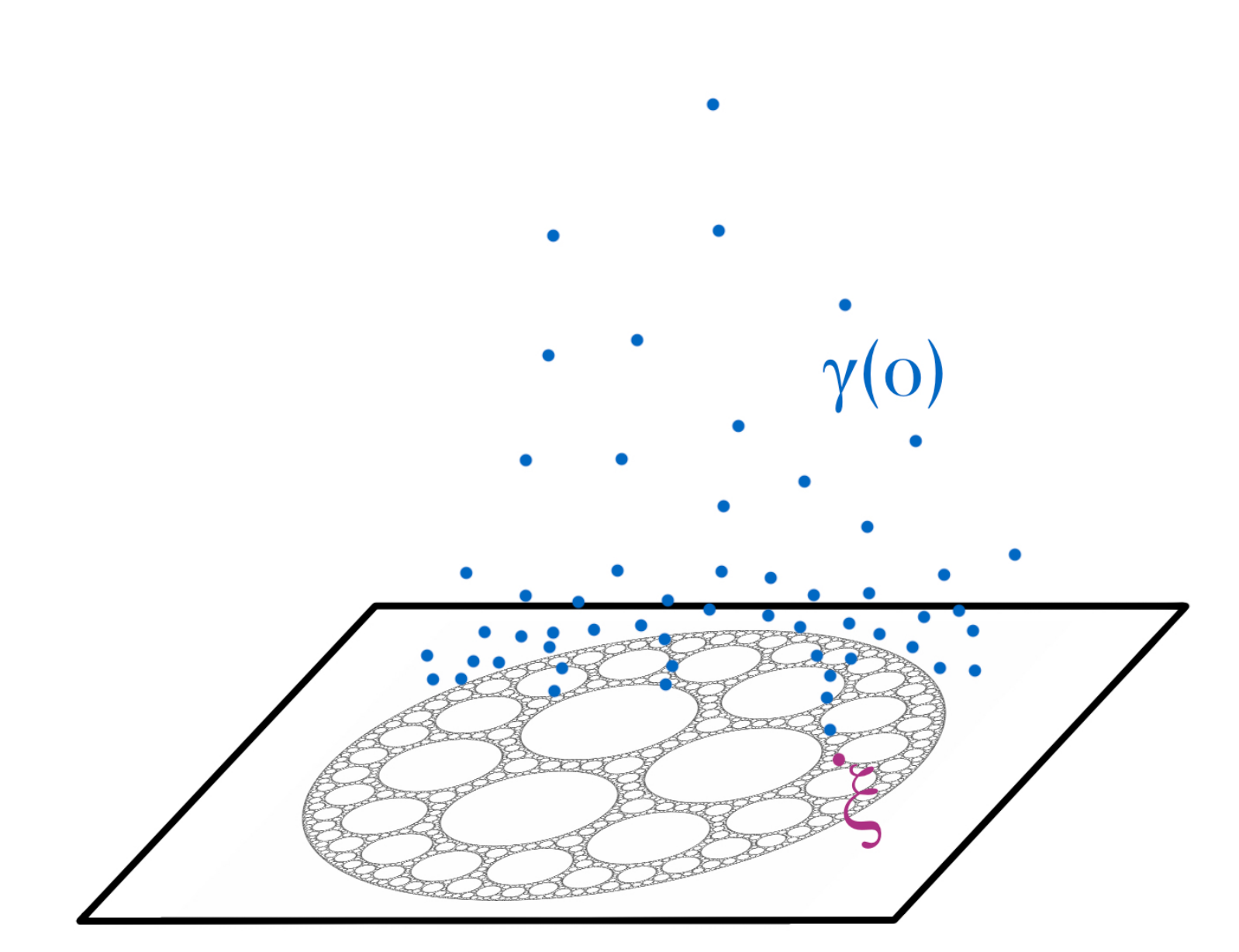}  
\end{center} \caption{Snow falling and limit sets}\label{snow}
\end{figure}    

For a  Kleinian group $\Gamma$, the quotient $\cM=\Gamma\ba \bH^3$ is a complete hyperbolic  manifold.
Conversely, any complete hyperbolic $3$-manifold $\cM$ arises as such a quotient for some Kleinian group $\Gamma$. Thus the study of hyperbolic $3$-manifolds is inseparable from the study of Kleinian groups. See Figure~\ref{snow} for a visual metaphor illustrating limit sets and domains of discontinuity.
The notions of limit set and convex core play a central role. 

\begin{definition}  The limit set $\Lambda_\Ga\subset \hc$  is  the set of all accumulation points of $\Gamma(o)$ for $o\in \bH^3\cup \hc$. Its complement $\Omega_\Ga=\hc-\La_\Ga$ is called the domain of discontinuity.
\end{definition}
We often omit the subscript $\Ga$ from $\La_\Ga$ and $\Omega_\Ga$ when 
the group under consideration is clear from context.
 A useful picture is to imagine snow falling: the limit set is precisely where the snow accumulates on the ground (see Figure \ref{snow}). When $\Ga< G$ is a lattice, that is, when $\Ga\ba \bH^3$ has finite volume, the snow covers the entire ground, i.e., $\La=\hc$. For non-lattice Kleinian groups, one often obtains breathtaking fractal patterns. See Figure \ref{limitex}.
The first  image is homeomorphic to a circle, the next two to the Sierpi\'nski carpet (the planar fractal obtained by repeatedly removing central squares) and the last to a Cantor set.

\begin{figure}[htbp] \begin{center}
  \includegraphics [height=4cm]{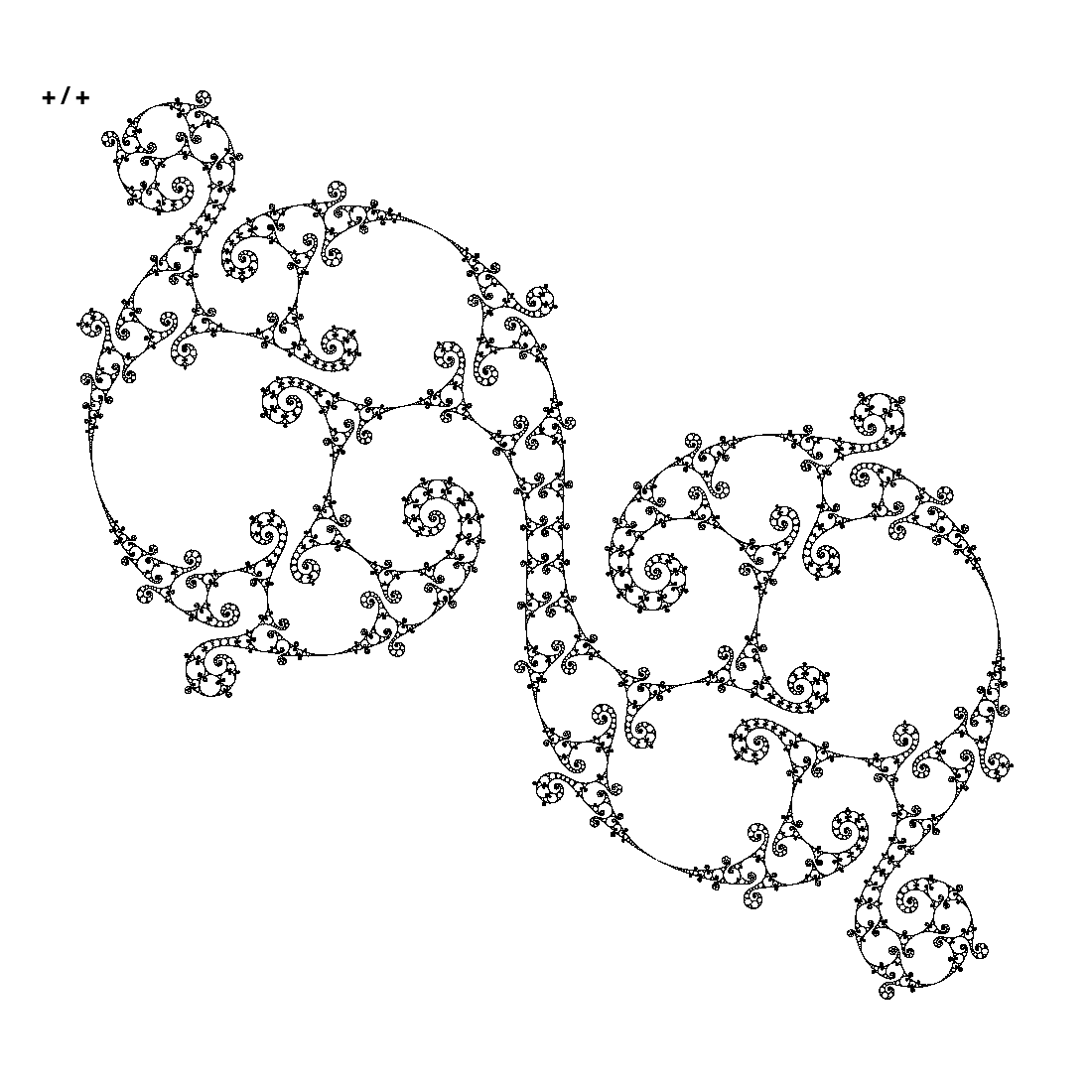}\quad
 \includegraphics [height=4cm]{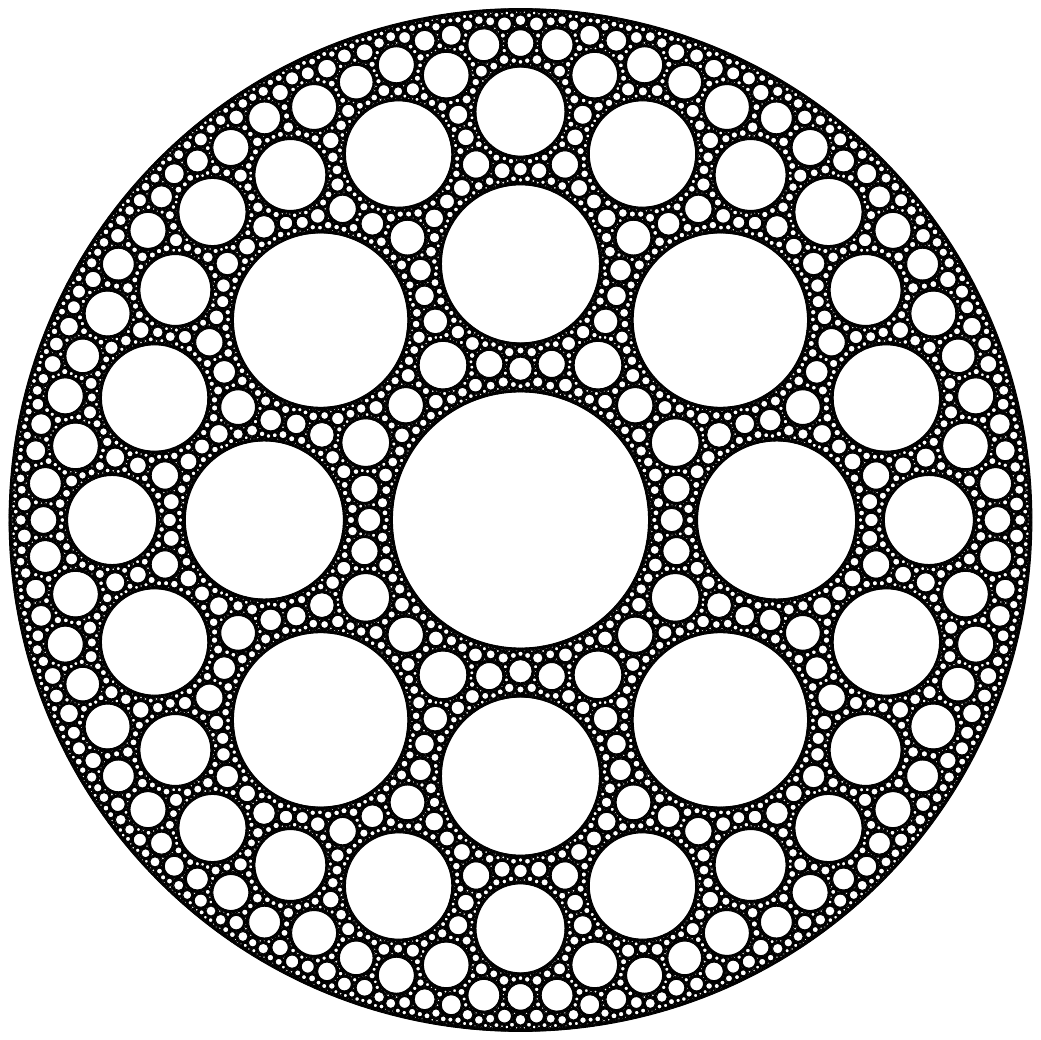}\quad
 \includegraphics [height=4cm]{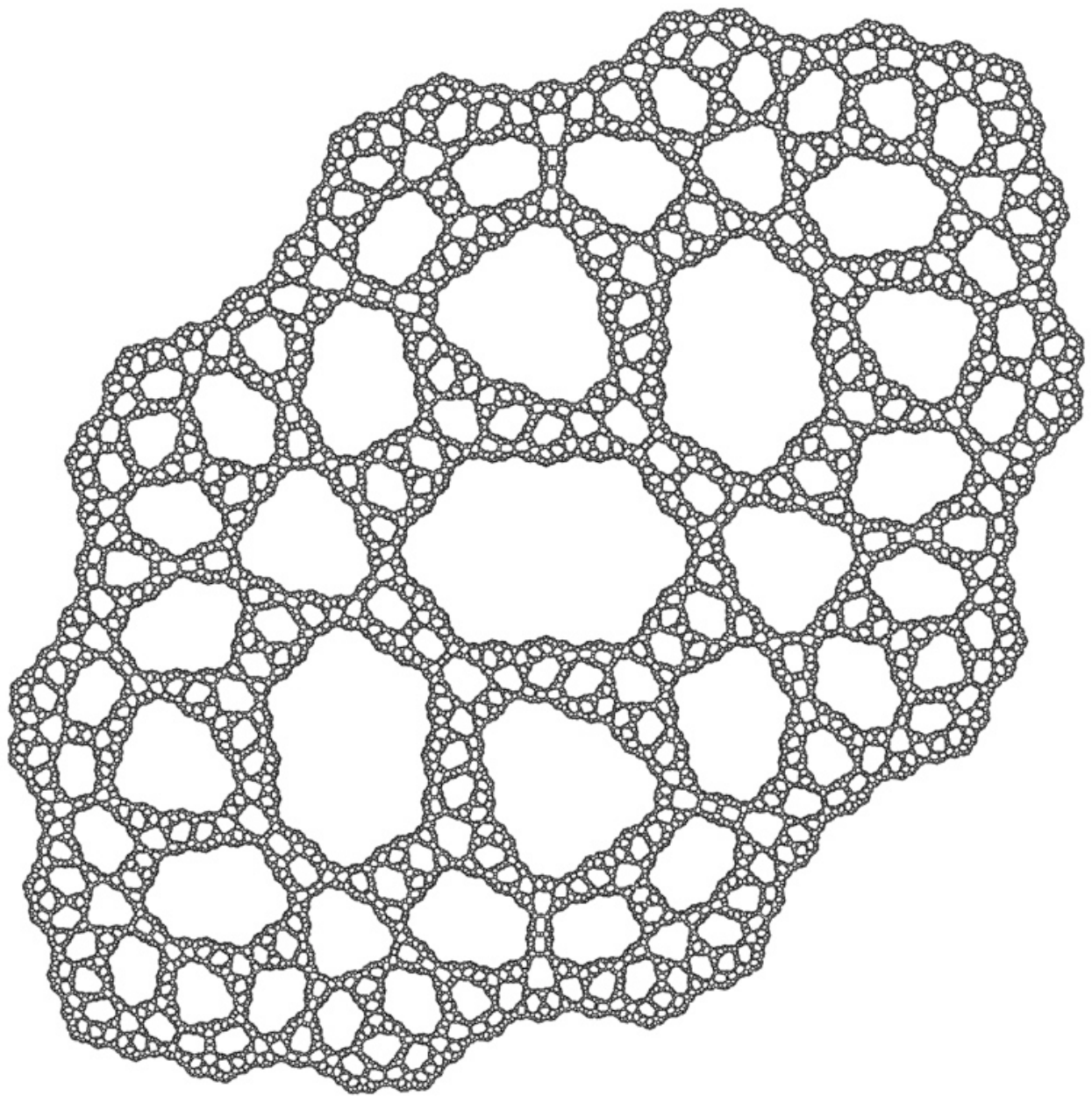}\quad 
 \includegraphics [height=4cm]{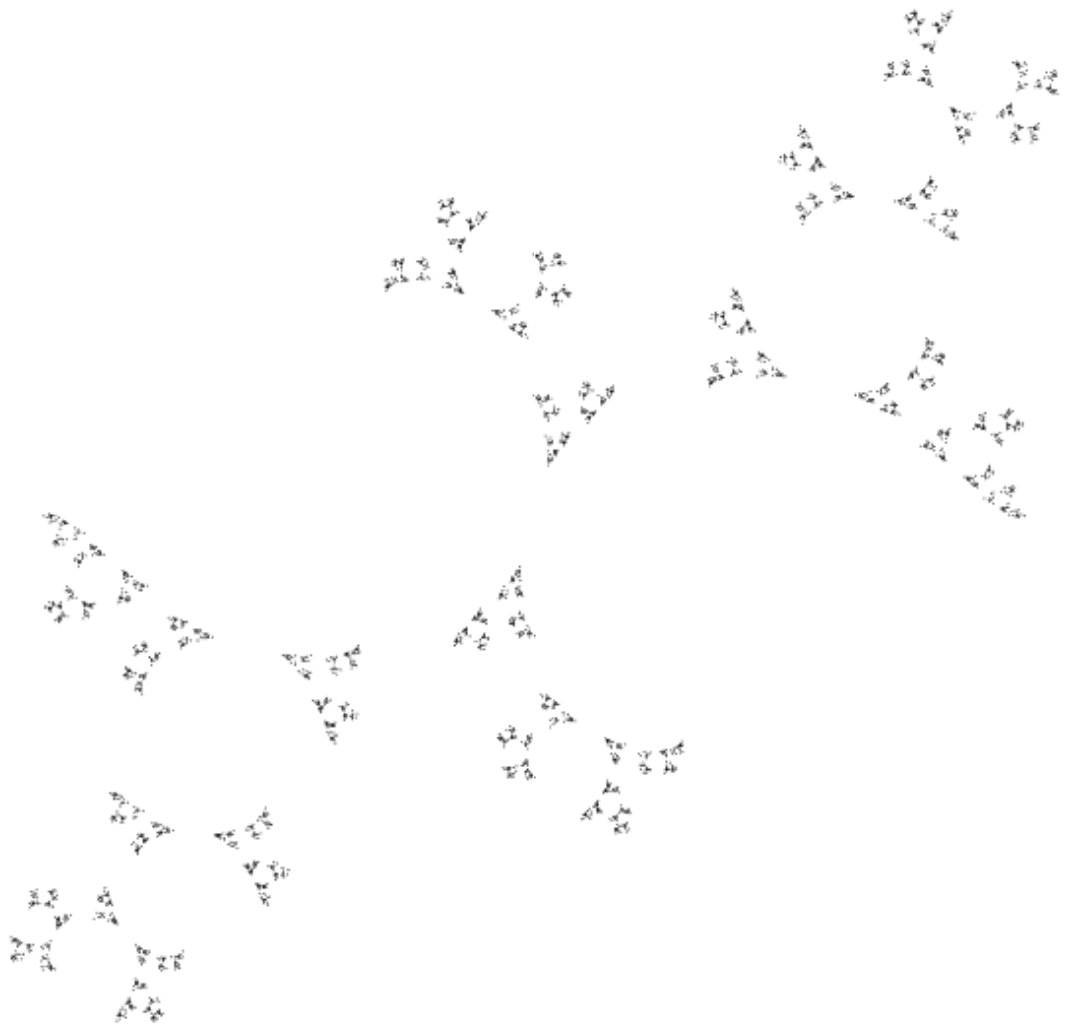}
 \end{center}\caption{Fractal limit sets of non-lattice Kleinian groups}\label{limitex}
\end{figure} 
The convex core of $\cal M=\Ga\ba \bH^3$ is 
 $$\op{core} \cM :=  \Gamma\ba \op{hull} \Lambda \; \subset  \; \cM $$
 where $\op{hull} \Lambda \subset \bH^3$ denotes  the convex hull of $\Lambda$. See Figure~\ref{convexcore} for an illustration of the convex core.

\begin{figure}[htbp] \label{convexcore}\begin{center}
    \includegraphics[height=4cm]{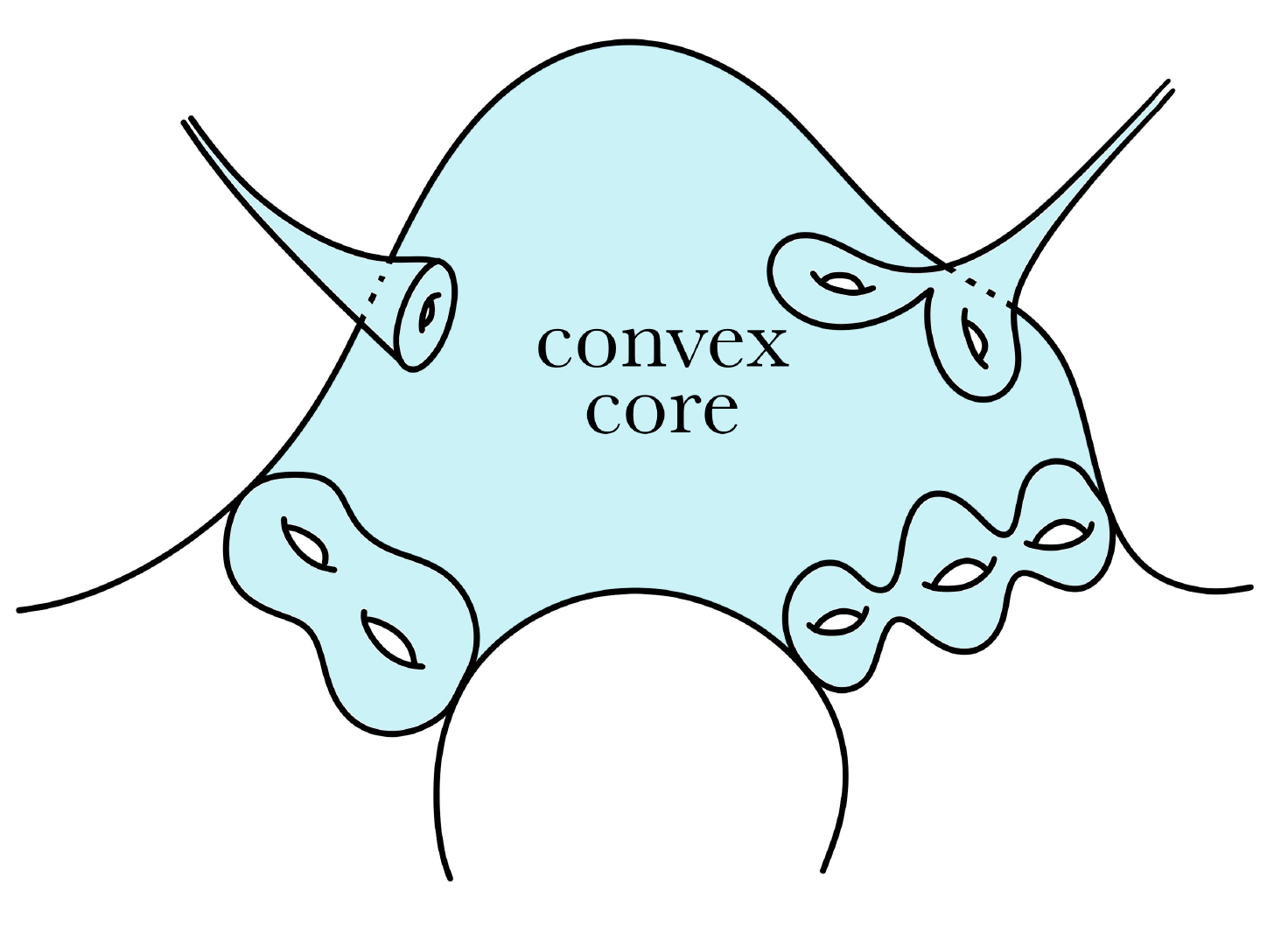}
\end{center} \caption{Convex core of a geometrically finite hyperbolic $3$-manifold}
\end{figure}    

\begin{definition}\label{def}
A Kleinian group $\Gamma$ is {\it geometrically finite} if the unit neighborhood of $\core{\cM}$ has finite volume. If $\core{\cM}$ is compact, we call $\Ga$ {\it convex cocompact}.
 \end{definition}
Geometrically finite (respectively, convex cocompact) groups are natural generalizations of  lattices 
(respectively, cocompact lattices). Among geometrically finite groups, lattices are precisely those whose limit sets are the whole $\hc$, whereas the limit sets of all others have Hausdorff dimension strictly smaller than $2$. 

While there are only countably many lattices in $G$, up to conjugation, by the local rigidity theorem of Selberg and Weil,  the Bers-Sullivan-Thurston density conjecture (now a theorem due to Namazi-Souto and Ohshika building on the work of many others (see \cite[§5.9]{Marden2016hyperbolic})) asserts that geometrically finite groups form an open and dense subset of the space of all finitely generated Kleinian groups. Therefore, results proved for geometrically finite groups apply to a very broad class of Kleinian groups.

\subsection{Circle-counting theorem for geometrically finite Kleinian groups} 
A fundamental observation of Sullivan \cite{Sullivan1979density} inspired by Patterson \cite{Patterson1976limit} links the action of $\Ga$ on $\bH^3$ with its action on the boundary $\hc$. In particular,  the growth rate of a $\Ga$-orbit in $\bH^3$ is determined by the Hausdorff dimension $\op{dim} \La$. 

\begin{theorem} [\cite{Sullivan1979density}, \cite{Sullivan1984entropy}] For any geometrically finite group
$\Ga<\PSL_2(\c)$,
$$\delta =\op{dim}(\La) $$
where $\delta=\delta(\Ga)$
denotes the critical exponent
$$\delta:= \limsup_{T\to \infty}\frac{1}{T} {\log \#\{x\in \Ga (o): d(x, o)\le T\}}, \quad o\in \bH^3. $$
\end{theorem}

We call a circle packing\footnote{Usually, a circle packing refers to a collection of circles with disjoint interiors that cannot be enlarged without intersecting the interiors of others. In this article, however, a circle packing simply means a countable union of circles.} $\cal P$ locally finite if, for any bounded region $B\subset \c$ and any $\e>0$, there are only finitely many circles in $\cal P$ with radii greater than $\e$ that intersect $B$.
This condition is necessary in order to pose a circle counting problem. Theorems \ref{ko} and \ref{os} are special cases of the following joint work with Shah:

\begin{theorem} [\cite{OS_inv}]\label{os2}Let $\P$ be a locally finite circle packing invariant under a geometrically finite Kleinian group $\Ga$ and with finitely many $\Ga$-orbits.
  If $\delta \le 1$, we further assume that $\cal P$ does not contain an infinite bouquet of tangent circles (see Figure \ref{bou}). Then there exist a constant $0<c_{\cal P}<\infty $ and a locally finite measure $\omega_\Ga$ on $\La\cap \c$ such that
  for any region $R\subset \c$ bounded by a piecewise algebraic curve\footnote{The piecewise algebraic boundary condition can be replaced by the weaker assumption $\omega_\Ga(\partial R)=0$.},
  $$\#\left\{C\in \cal P: \op{rad} (C) \ge \tfrac{1}t,\; C\cap R\ne \emptyset  \right\} \sim c_{\cal P} \cdot \omega_\Ga (R) \cdot  t^{\delta}\quad\text{as $t\to \infty$.}$$ 
\end{theorem}

\begin{figure}[htbp] \begin{center}
    \includegraphics[height=2cm]{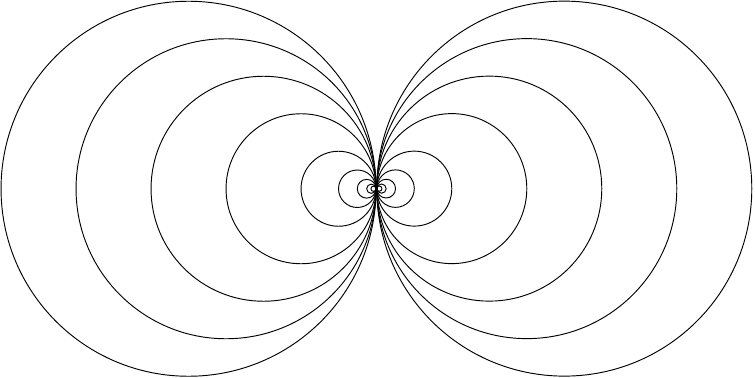}
\end{center} \caption{Infinite bouquet of tangent circles}\label{bou}
\end{figure}    

\subsection{Decoding the constant and the measure} In results of this type, it is often the constant $c_{\cal P}$ and the measure $\omega_\Ga$ that encode the complexity of the problem and reflect the methods involved in the proof. 
$$\text{Where do the constant $c_{\cal P}$ and the measure $\omega_\Ga$ come from? }$$ Their existence rests on fundamental results in the dynamics of hyperbolic manifolds $\cM=\Ga\ba \bH^3$, beginning with the construction of a geometric measure on the limit set.

To motivate the definition of a “geometric” measure, consider how the spherical measure on $\hc$ transforms under the action of $G$. Fix $o=(0,0,1)\in \mathbb H^3$ and let $K=\Stab_G(o)$. Let $m_o$ denote the $K$-invariant probability measure on $\hc$. For $g\in G$, the pushforward\footnote{$g_*m_o(E)=m_o(g^{-1} (E))$ for all Borel subsets $E\subset \hc$.}  $g_*m_o$ is absolutely continuous with respect to $m_o$, with Radon–Nikodym derivative
\be\label{bus} \frac{dg_*m_o}{dm_o} (\xi)=e^{ 2\beta_\xi (o,go)} \quad\text{ for all $g\in G$},\ee 
where $\beta$ denotes the Busemann function. Recall that for $\xi\in \hc$ and $x,y\in \bH^3$,
$\beta_\xi(x,y)=\lim_{t\to +\infty} d(\xi_t, x)- d(\xi_t, y)$ for a geodesic ray $\xi_t$ toward $\xi$, where $d$ denotes the hyperbolic distance in $\bH^3$. The Busemann function encapsulates
the conformal transformation law for the action of $G$ on $\hc$.

The spherical measure $m_o$ has full support on $\hc$.
To study a Kleinian group $\Gamma$, however, one needs measures supported on its limit set $\Lambda$. Patterson and Sullivan introduced such measures, which capture the asymptotic distribution of $\Gamma$-orbits. For geometrically finite groups, their construction yields a unique conformal measure of the correct dimension. A measure $\nu$ on $\hc$ is called a $\Ga$-conformal measure of dimension $s$ if it satisfies
\be \label{conformal} \frac{d\ga_*\nu}{d\nu} (\xi) =e^{s \beta_{\xi}(o,\ga o)} \quad\text{ for all $\ga\in \Ga$.}\ee

\begin{theorem}[\cite{Patterson1976limit}, \cite{Sullivan1979density}]\label{psm}For any geometrically finite group $\Ga<\PSL_2(\c)$,
there exists a unique $\Ga$-conformal probability measure $\nu_o$ on $\hc$ of dimension $\delta$.
Moreover, $\nu_o$ is supported on $\La$. \end{theorem} 
This measure is called the Patterson–Sullivan measure, or the geometric measure associated to $\Gamma$. The measure $\omega_\Ga$ in Theorem \ref{os2} is
given explicitly as $$d\omega_\Ga(z) =(|z|^2+1)^{\delta} d\nu_o(z), \quad z\in \La\cap \c.$$ When $\Ga$ is convex cocompact or the symmetry group of an Apollonian circle packing, the measures $\nu_o$ and $\omega_\Ga$ coincide (up to scaling) with the $\delta$-dimensional Hausdorff measures with respect to the spherical metric and the Euclidean metric on $\c$, respectively.

 Explaining the constant $c_{\cal P}$ requires two key ingredients:
the Bowen-Margulis-Sullivan (BMS) measure on the unit tangent bundle $\T^1(\cM)$, and the skinning measure on a properly immersed geodesic plane in $\cM$.

\subsubsection{Bowen-Margulis-Sullivan measure} A tangent vector $v\in \T^1(\bH^3)$ is determined by the forward and backward endpoints $v^+, v^-\in \hc$ of the associated geodesic, together with the Busemann value $\beta_{v^+}(o, v)$. This gives the Hopf-parametrization
\be\label{hopf} \op{T}^1(\bH^3) =\left( \hc \times \hc -\text{diag}\right) \times \br .\ee 
In these coordinates, the Bowen-Margulis-Sullivan measure $m^{\BMS}$ on $\T^1(\cM)$ is induced from the $\Ga$-invariant measure $\tilde m^{\BMS}$ on $\T^1(\bH^3)$, defined by incorporating the conformal factors from \eqref{conformal} into
the product of Patterson-Sullivan measures and the Lebesgue measure on $\br$:
\begin{align*}
d \tilde m^{\BMS}(v)& =
e^{\delta \beta_{v^+}(o, v)}\; e^{\delta \beta_{v^-}(o,v) }\;d\nu_o (v^+) d\nu_o(v^-) dt.
\end{align*}

 \begin{figure} \label{support} \begin{center}
  \includegraphics [height=4.5cm]{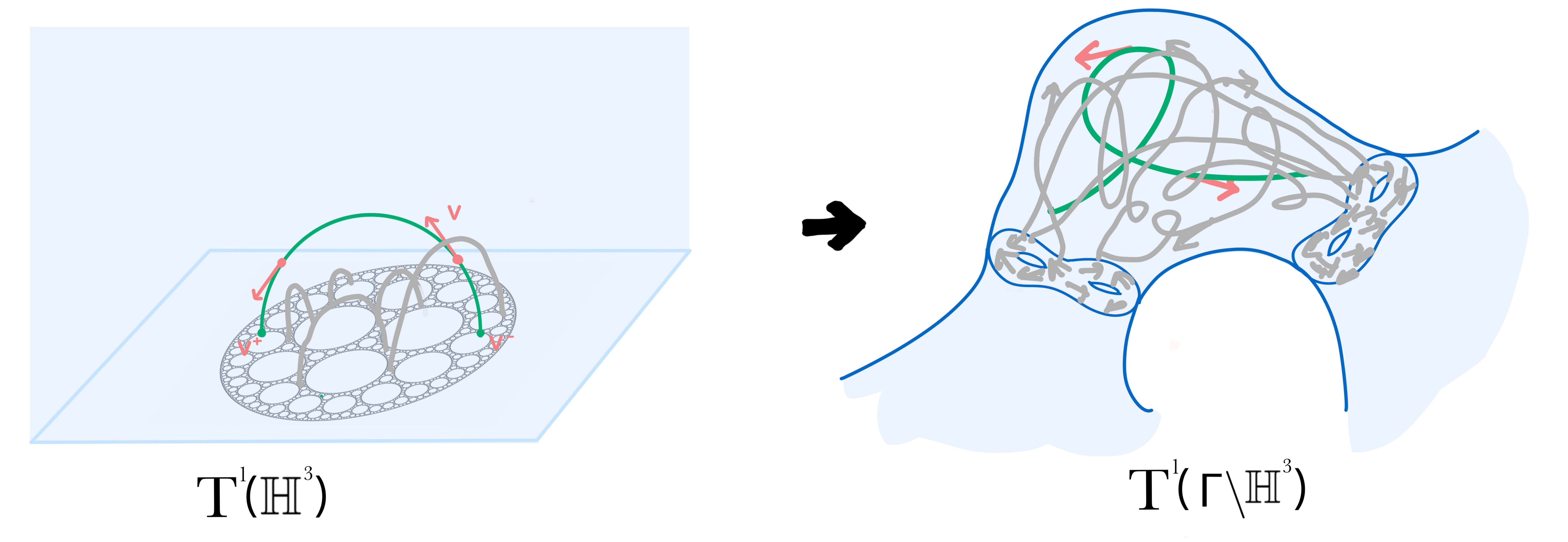} \end{center}
  \caption{Support of BMS measure}
\end{figure}

Sullivan proved that if $\Ga$ is geometrically finite, then $m^{\BMS}$ is a {\em finite} measure \cite{Sullivan1984entropy}:
$$|m^{\BMS}|<\infty .$$
The support of the BMS measure consists of all tangent vectors with endpoints in the limit set, and hence
projects into the convex core under the basepoint projection $\T^1(\cM)\to \cM$ (cf. Figure \ref{support}). For convex cocompact groups, this makes finiteness immediate. In the presence of cusps, that is, for geometrically finite but non-convex-cocompact groups, finiteness becomes a non-trivial result, reflecting the specific geometry of the cusps.

This finiteness is crucial: although $\cM$ has infinite Riemannian volume,
the essential dynamics of flows is captured by this finite BMS measure. This explains why many central theorems
in dynamics for hyperbolic manifolds are formulated for geometrically finite quotients.

\subsubsection{Skinning measure of a properly immersed geodesic plane} 
 A circle $C\in \cal C$ corresponds uniquely to a geodesic plane $C^\dagger\subset \bH^3$ (the  hemisphere above $C$) and vice versa (Figure \ref{geop}).
The orbit $\Ga C$ is closed in $\cal C$ if and only if the inclusion map $\op{Stab}_\Ga (C^\dagger)\ba  C^\dagger\to \cM$ is proper  \cite{OS_inv}.
In joint work with Shah \cite{OS_Jams},
we introduced the skinning measure for such $C^\dagger$. It is defined by assigning Patterson-Sullivan weights on the forward and backward endpoints of geodesics orthogonal to $C^\dagger$. Concretely,
$$d\mu^{\mathsf{sk}}_{C^\dagger}(v)= e^{\delta \beta_{v^+}(o, v)} d\nu_o(v^+)
+ e^{\delta \beta_{v^-}(o, v)} d\nu_o(v^-) $$
where $v^+$ and $v^-$ are the visual images of $v$, with $v$ taken as outward normal to $C^\dagger$. The normalization ensures that this measure is invariant under $\op{Stab}_\Ga (C^\dagger)$, and hence descends to a locally finite measure on the properly
immersed geodesic plane $\op{Stab}_\Ga (C^\dagger)\ba  C^\dagger\subset \cM$.

\begin{figure}[htbp] \begin{center}
  \includegraphics [height=5cm]{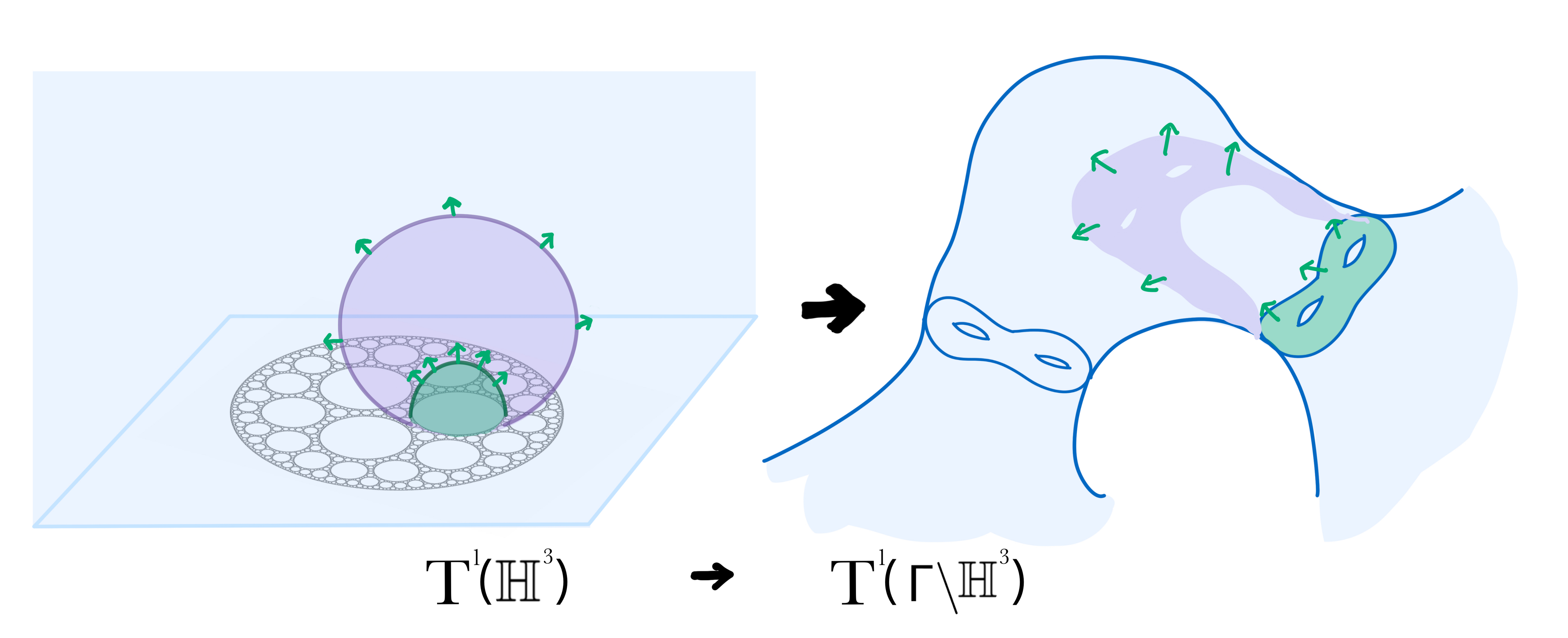}\end{center}
  \caption{Orthogonal translates of geodesic planes}\label{ot}
\end{figure}    

Intuitively the skinning measure records how the geodesic plane $C^\dagger$ intersects the limit set through its normal vectors. The skinning constant of $C$ is the total mass of this measure:
$$0<\mathsf{sk}(C):=  \mu^{\mathsf{sk}}_{C^\dagger}( \op{Stab}_\Ga (C^\dagger)\ba  C^\dagger)\le \infty  .$$

 For a locally finite circle packing $\cal P=\Ga C_1\cup \cdots \cup \Ga C_\ell$, we set
$ \mathsf{sk}_\Ga(\cal P)= \mathsf{sk}(C_1)+\cdots +  \mathsf{sk}(C_\ell )  .$
This constant is finite
whenever $\delta>1$ or $\cal P$ has no infinite bouquet of tangent circles. Finally
the constant $c_{\cal P}$ in Theorem \ref{os2} is given
by $$0<c_{\cal P}= \frac{\mathsf{sk}_\Ga (\cal P)}{|m^{\BMS}|}<\infty.$$

With these finiteness results in place, the key technical ingredient in the proof of Theorem \ref{os2} is the description of the asymptotic distribution of orthogonal translates of properly immersed geodesic planes under the geodesic flow (Figure \ref{ot}), using the local mixing of the geodesic flow\footnote{The original proof of Theorem \ref{ko} used the Descartes circle theorem: for Apollonian circle packings, it allows one to interpret the circle-counting problem as a point-counting problem in the space of horospheres. This approach is specific to the Apollonian case.}(Theorem \ref{win}).
This distribution is governed by the skinning measure, which serves as the bridge between the geometry of circle packings and the dynamics of flows in infinite volume.

\section{Orbit closures of circles and  rigidity of geodesic planes}\label{s:mmo}
\subsection{Orbit closures in round Sierpi\'nski carpets}
In the previous section, the circle counting problem  concerned 
 closed $\Ga$-orbits of circles. In this section, we study all possible orbit closures of arbitrary 
circles under $\Gamma$ in $\cal C$. When $\Ga C$ is not closed, what can its closure in $\cal C$ look like? For example, 
what is the closure of the orbit of the red circle in Figure \ref{red} under the symmetry group $\Ga$?

\begin{figure}[htbp] \begin{center}
    \includegraphics [height=3.5cm]{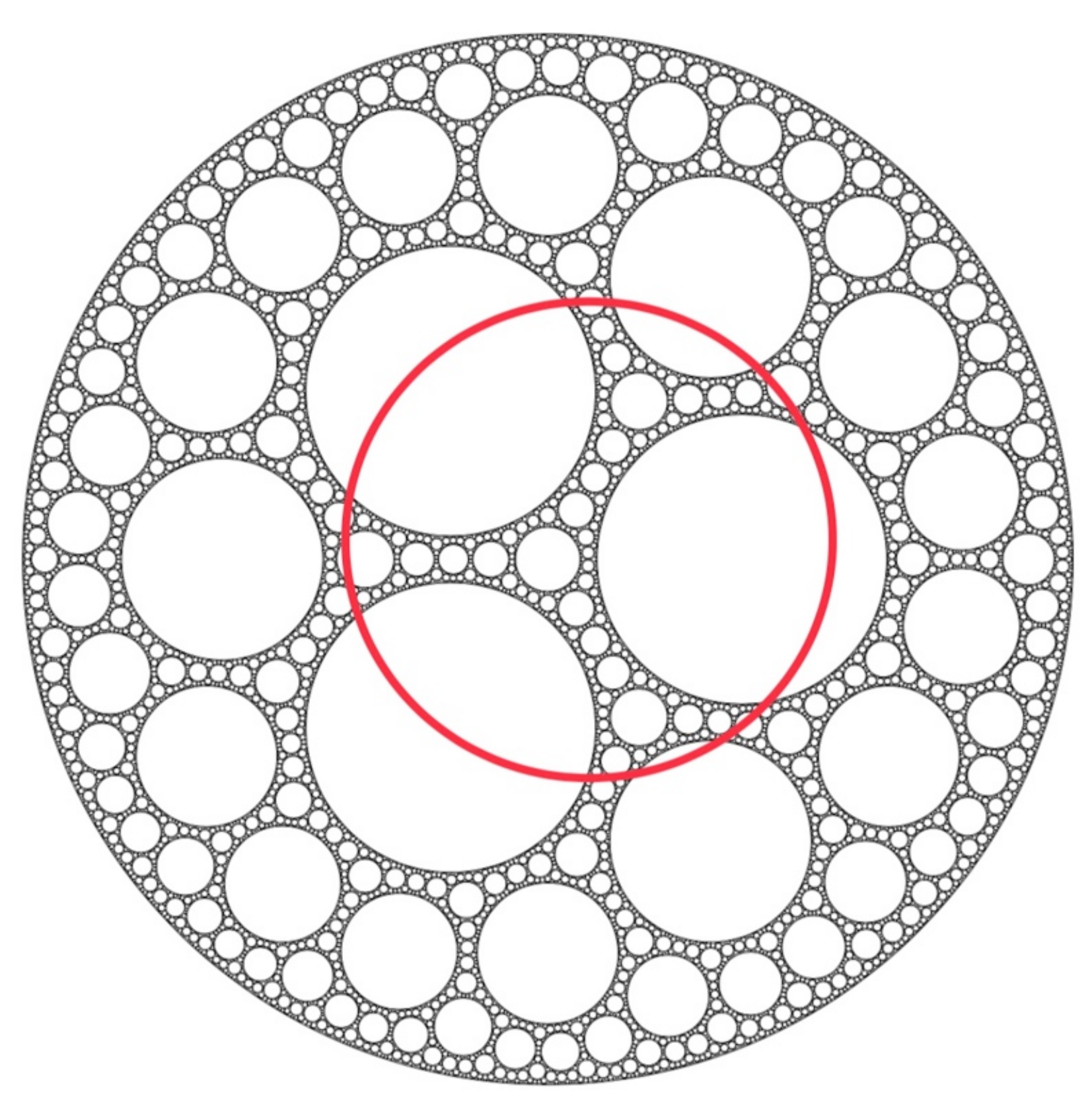}  \end{center}
 \caption{Round Sierpi\'nski carpet}\label{red}
\end{figure}    

If $C$ does not intersect the limit set $\La$, then the orbit $\Ga C$ is closed in a trivial way
and is uninteresting for our purposes.
Hence our ambient space will be the space of circles intersecting $\La$, which we denote by
$$\cal C_\La\subset \cal C.$$ Clearly $\Ga (\cal C_\La)=\cal C_{\La}$.
We further define $$\cal C_\La^*\subset \cal C_\La$$ to be the family of circles separating $\La$; that is,
$C\in \cal C_\La^*$ if both the interior and exterior of $C$ intersect $\La$. Whether $C$ lies in $\cal C_\La^*$ or not influences the behavior of its orbit $\Ga C$.

An important class of Kleinian groups in this discussion consists of those whose limit sets $\Lambda$ satisfy:
$$\Omega=\hc-\La=\bigcup B_i$$
where $B_i$ are infinitely many round disks with mutually disjoint closures.
Such a limit set is called a {\it round Sierpi\'nski carpet}. We give a
complete description of orbit closures
 for geometrically finite groups whose limit sets are round Sierpi\'nski carpets. This classification was first established in joint work with McMullen and Mohammadi \cite{MMO_inventiones} for convex cocompact groups, and was later extended to all geometrically finite groups in joint work with Benoist \cite{BO_etds}.

\begin{theorem}[\cite{MMO_inventiones}, \cite{BO_etds}]\label{mmo1}Let $\Ga<\PSL_2(\c)$ be a geometrically finite group whose limit set $\La$ is a round Sierpi\'nski carpet. Let $C\in \cal C_\La$.
\begin{itemize}
    \item If $C\in \cal C_\La^*$, then $\Ga C$ is either closed or dense in $\cal C_\La$.
\item If $C \notin \cal C_\La^*$, then $\overline{\Ga C}=\{D\in \cal C_\La: D\subset \Ga \overline{B}\}$ where $B$ is the component of $\Omega$ whose closure contains $C$.
\end{itemize}
\end{theorem}
\begin{figure}[htbp] \begin{center}
 \includegraphics [height=4cm]{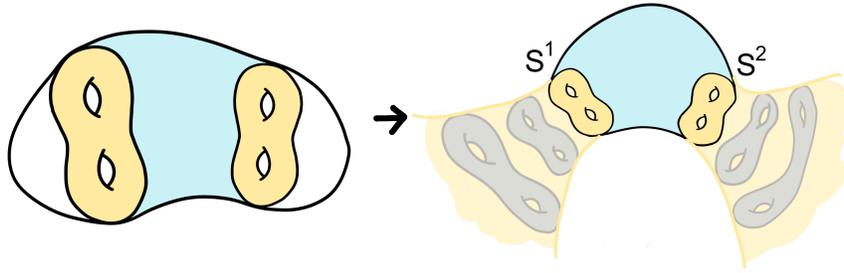}
\end{center}\caption{A manifold with round Sierpiński carpet limit set: the convex core has totally geodesic boundary}\label{cut}
  \end{figure}

Geometrically finite hyperbolic $3$-manifolds $\cM$ with round Sierpiński carpet limit sets are precisely those whose convex core has compact totally geodesic boundary. Each $\Ga$-orbit of a component $B$ of $\Omega$ corresponds to a boundary component of the convex core, and $B$ being a round disk means that the associated boundary surface is totally geodesic.
In fact, $\core \cM$ can be obtained by removing finitely many disjoint compact embedded totally geodesic surfaces from 
a complete finite-volume hyperbolic $3$-manifold, as illustrated in Figure~\ref{cut}, where two disjoint compact geodesic surfaces $S_1$ and $S_2$ are removed from a closed hyperbolic $3$-manifold and become the convex core boundary of the resulting $\cM$.
It follows that, up to isometry, there are only countably many such manifolds, since there are only countably many finite-volume hyperbolic $3$-manifolds and only countably many ways to remove finitely many compact geodesic surfaces.

\subsection{Geometric structure and limit sets of acylindrical hyperbolic $3$-manifolds} We established a similar orbit-closure dichotomy for a much broader class of geometrically finite groups, called {\it acylindrical} groups. 
A compact core of $\cM$ is a compact connected submanifold $N$ with boundary whose inclusion  induces an isomorphism $\pi_1(N)\simeq  \Ga$, unique up to homeomorphism. 
A geometrically finite manifold $\cM =\Ga\ba \bH^3$ is called {\it acylindrical} if its compact core $N$ satisfies: 
\begin{itemize}
    \item  $N$ has incompressible boundary;
    \item  every essential cylinder in $N$ is boundary parallel. 
    \end{itemize}
    If moreover the boundary of $\core \cM$ is totally geodesic, we call $\cM$ {\em rigid acylindrical}. Note that acylindricity is a purely topological condition: any geometrically finite hyperbolic $3$-manifold homeomorphic to an acylindrical one is itself acylindrical.
    It turns out that any geometrically finite acylindrical hyperbolic $3$-manifold arises as a quasiconformal deformation of a rigid acylindrical one.
To state this precisely,
recall that an orientation-preserving homeomorphism $f:\hc\to \hc$ is called $\kappa$-quasiconformal for some $\kappa\ge 1$ if it maps infinitesimal circles to ellipses of eccentricity at most $\kappa$; that is,
  \be\label{qc} \limsup_{r\to 0}\frac{\sup \{|f(x)-f(z)|: |z-x|=r\}}{\inf \{|f(x)-f(z)|: |z-x|=r\}} \le \kappa. \ee 
  The $1$-quasiconformal maps are conformal and hence they are elements of $\PSL_2(\c)$.
Two Kleinian groups $\Ga_1$ and $\Ga_2$ are said to be quasiconformally conjugate if they are conjugated by such a map, i.e., there exists an isomorphism $\rho:\Ga_1\to \Ga_2$ and a quasiconformal homeomorphism $f$ of
$\hc$ such that for all $ \gamma\in \Ga_1$,
$\rho(\ga)= f\circ \ga \circ f^{-1}$ on $\hc$. Such a representation is called a quasiconformal deformation.

By the Ahlfors-Bers theorem, quasiconformal deformations of a geometrically finite acylindrical group, considered up to conjugation, are parametrized by the product
$$\prod_i\mathrm{Teich} (S_i)$$
where $S_i$ are the components of the boundary of the convex core of $\Ga\ba \bH^3$ and $\mathrm{Teich}(S_i)$ denotes the Teichm\"uller space of $S_i$,  which parametrizes its quasiconformal deformations. The following result of Thurston and McMullen provides a distinguished representative in each such quasiconformal class:
\begin{theorem}[\cite{Thurston1986hyperbolic}, \cite{Mc_iter}] Every geometrically finite acylindrical Kleinian group is quasiconformally conjugate to a geometrically finite  rigid acylindrical group, unique up to conjugation.
\end{theorem}

 \begin{figure}[htbp] \begin{center}
  \includegraphics [height=4.5cm]{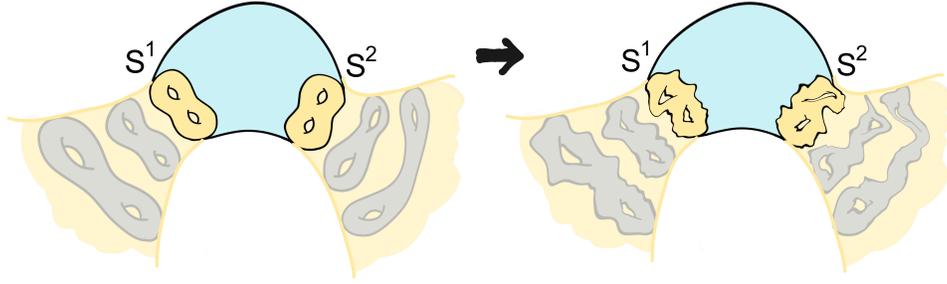}
   \end{center}
     \caption{A quasiconformal deformation of a convex cocompact hyperbolic $3$-manifold with two ends}\label{two}
  \end{figure} 
In Figure \ref{two}, we illustrate this theorem for a convex cocompact hyperbolic $3$-manifold $\cM$  whose convex core has two boundary components $S_1$ and $S_2$. For any choice of $\rho_i\in \op{Teich}(S_i)$, $i=1,2$, there exists a quasiconformal deformation of the ambient $3$-manifold $\cM$
whose convex-core boundary has been deformed according to these parameters.

For convex cocompact groups,  acylindricity is equivalent to the condition that the limit set  $\La$ is a Sierpi\'nski carpet: 
$$\hc-\La=\bigcup B_i$$
where $B_i$ are infinitely many {\emph{Jordan}} disks with mutually disjoint closures. When the $B_i$ are round disks, $\La$ is a round Sierpi\'nski carpet, corresponding precisely to the rigid acylindrical case. For a  general geometrically finite acylindrical group, the closures of $B_i$ may intersect. We regroup the components of $\Omega$ into maximal families of Jordan disks forming trees: declare $B_i\sim B_j$ if their closures intersect, which spans an equivalence relation on the collection $\{B_i\}$. Let $T_\ell$ be the union of all disks in the same equivalence class. Then the closure of $T_\ell$ forms a "tree of disks", in the sense that the dual graph whose vertices corresponds to the disks and whose edges connect pairs of tangent disks is a tree.
Thus the limit set has the structure
$$\hc-\La=\bigcup T_\ell $$ where $T_\ell$ are  maximal trees of components with mutually disjoint closures (see the right image in Figure \ref{twol}).

 \begin{figure}[htbp] \begin{center}
   \includegraphics [height=4.5cm] {qfs}  
   \qquad\qquad 
    \includegraphics [height=4.5cm] {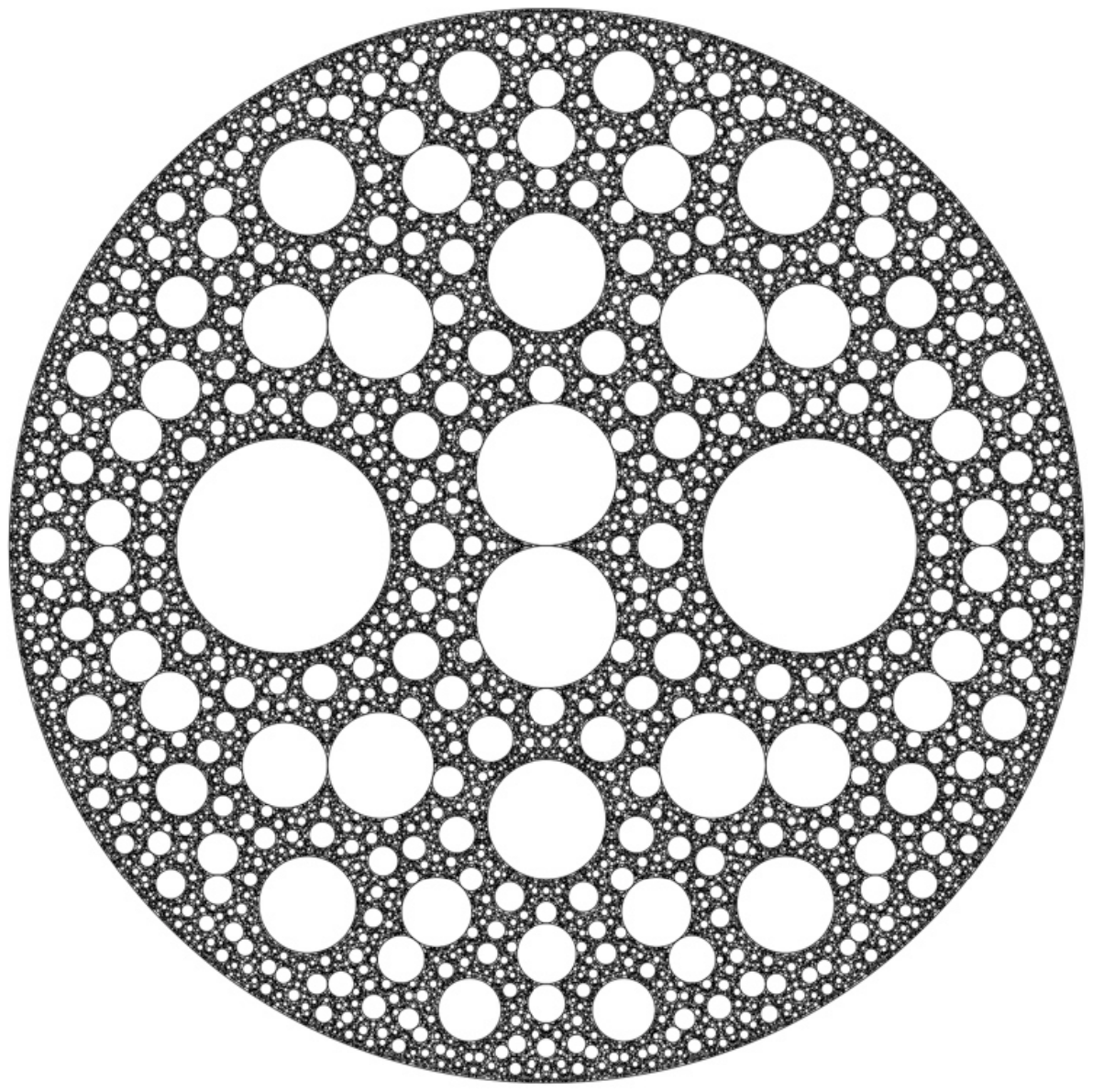}  
   \end{center}
     \caption{Limit sets of geometrically finite acylindrical groups: \cite{Zhang2022construction}}\label{twol}
  \end{figure}

\subsection{Orbit closures for geometrically finite acylindrical hyperbolic $3$-manifolds} With this picture in mind, we have the following orbit-closure classification,  proved in joint work with McMullen and Mohammadi \cite{MMO_duke} for convex cocompact groups, and later extended with Benoist \cite{BO_etds}.
 
\begin{theorem} [\cite{MMO_duke}, \cite{BO_etds}]\label{c2}
    Let $\Ga<\PSL_2(\c)$ be any geometrically finite acylindrical group.
     Then: \begin{enumerate}
 \item  Any $\Gamma$-invariant subset of $\mathcal C_\La^*$ is either a finite
union of closed
$\Gamma$-orbits or dense in $\mathcal C_\La^*$.

\item  There are at most countably many closed $\Gamma$-orbits in $ {\cal C}_\La^*$.
\item  If $\Ga$ is rigid, then any $\Gamma$-invariant subset of $\mathcal C_\La^*$ is either a finite union of closed
$\Gamma$-orbits or dense in $\mathcal C_\Lambda$.
\end{enumerate}
\end{theorem}Note that (1) not only includes the closed-or-dense dichotomy for individual orbits
$\Ga C\subset \cal C_\La^*$, but also implies that any infinite collection of closed $\Ga$-orbits is automatically dense in $\cal C_\La^*$.
  
\begin{remark} When $\Ga$ is not rigid, the closedness of $\Ga C$ in $\cal C_\La^*$ does not necessarily imply that it is closed in $\cal C_\La$. Indeed,
Zhang \cite{Yong_exotic} constructed a convex cocompact acylindrical group $\Ga$ and a circle $C\in \cal C_\La^*$ such that $\Ga C$ is closed in $\cal C_\La^*$ but accumulates in $\cal C_\La$.
This reflects the phenomenon that while the interior of $\core \cM$ behaves more like a homogeneous space, the interior together with its boundary is not. 
\end{remark}

 We also remark that possibly except for finitely many $\G$-orbits, all circles in $\cal C_\La$ meet $\Omega$:
\begin{theorem} [\cite{MMO_inventiones}, \cite{BO_etds}]\label{finite}If $\Ga$ is a geometrically finite non-lattice Kleinian group, 
then the set $\{C\in \cal C: C\subset \La\}$ consists of finitely many closed $\Ga$-orbits.
\end{theorem}

\subsection{Conjecture on orbit closures on Apollonian circle packing} The Apollonian group (the symmetry group of an Apollonian circle packing) almost falls into this framework but just misses, because its compact core is a genus-two handlebody and hence the boundary is compressible. Therefore it is not acylindrical.
We propose the following:
\begin{conjecture}[Orbit-closure dichotomy for Apollonian circle packings]
Let $\mathcal P$ be an Apollonian circle packing with symmetry group $\Gamma<\PSL_2(\mathbb C)$. For any circle $C\in \mathcal C_\Lambda^*$, the orbit $\Gamma C$ is either closed or dense in $\mathcal C_\Lambda$.
\end{conjecture}

Since Apollonian circle packings are among the most classical and visually striking examples of circle packings, this is a particularly natural and compelling problem.

\subsection{Topological rigidity of geodesic planes} Recall that any geodesic plane in $\bH^3$ is uniquely determined by a circle in $\hc$ and vice versa (see Figure \ref{geop}).
Via this correspondence,  Theorem \ref{c2} yields the topological rigidity of geodesic planes in the hyperbolic manifold $\cM=\Ga\ba \bH^3$. 
An immersed geodesic plane $P$ in $\cM$ is the image of a geodesic plane in $\bH^3$ under the quotient map $\bH^3\to \Ga\ba \bH^3$. Denote by $\cM^*$ the interior of the convex core of $\cM$.
A circle $C$ separates $\La$ if and only if the corresponding geodesic plane $P$ intersects $\cM^*$.
Theorems \ref{mmo1} and \ref{c2} imply the following:

\begin{theorem}[\cite{MMO_inventiones}, \cite{MMO_duke}, \cite{BO_etds}]
\label{m2} Let $\cM$ be a geometrically finite  acylindrical  hyperbolic $3$-manifold.
Then 
\begin{enumerate}
    \item  Any geodesic plane $P$  intersecting $\cM^{*}$ is either closed or dense in $\cM^{*}$. 
\item There are at most countably many properly immersed geodesic planes intersecting $\cM^*$.
\item If $\cM$ is rigid, 
then any geodesic plane $P$ intersecting $\cM^{*}$ is either closed or dense in $\cM$.
\end{enumerate}
\end{theorem}

\subsection{Thick recurrence of unipotent flows behind geodesic plane rigidity} We now explain the dynamical mechanism underlying geodesic plane rigidity for geometrically finite acylindrical groups. To bring in homogeneous dynamics, we pass to the frame bundle
$$\op{ F}(\cM)=\Ga\ba \PSL_2(\c) .$$
Any geodesic plane in $\cM$ is the image of a $\PSL_2(\br)$-orbit in $\Gamma\backslash \PSL_2(\c)$ under the basepoint projection $\mathcal F(\cM)\to \cM$. Hence, the classification of closures of geodesic planes follows from the classification of the closures of ${\PSL_2(\br)}$-orbits in $\Gamma\backslash \PSL_2(\c)$. A pivotal point is that $\PSL_2(\R)$ is generated by unipotent subgroups (the strict upper and lower triangular subgroups).
Recall that a matrix is unipotent if all of its eigenvalues are equal to $1$, and that unipotent subgroups are those consisting entirely of unipotent elements.

When $\Ga\ba \PSL_2(\c)$ has finite volume, any $\PSL_2(\R)$-orbit is either
closed or dense.  This is a special case of Ratner's theorem (Theorem \ref{ra}) proved using unipotent dynamics. The main difficulty in the infinite-volume case is the lack of recurrence of unipotent orbits: without recurrence, orbits rarely return to compact sets, and accumulation behavior becomes hard to capture.

The key to carrying out homogeneous dynamics in the acylindrical case is the construction of a compact subset $\mathcal R\subset \Gamma\backslash G$ such that for any $x\in \mathcal R$, the orbit $xU$ of a one-parameter unipotent subgroup $U=\{u_t:t\in \br\}<\PSL_2(\R)$ has "thick recurrence". That is, there exists $\kappa>1$ such that for any $s>0$, one can find $t\in \br $ with $s<|t|<\kappa s$ satisfying $xu_t\in \mathcal R$. Moreover, every $\PSL_2(\R)$-orbit corresponding to a geodesic plane intersecting $\cM^*$ arises from some $x\in\mathcal R$. The construction of $\cal R$ relies on the {\emph{positive modulus property}} of the Sierpi\'nski limit set:  $\inf_{\ell\ne k} \text{mod}(\hc- (\overline T_\ell \cup \overline T_k))>0;$
  where, for an annulus $A$, the $\text{mod} (A)=\log r$ if $A$ is conformally equivalent to the round annulus  $\{z\in \c: 1<|z|<r\}$.
This condition is used to show that every separating circle $C$ has a {\it thick} circular slice  $C\cap \Lambda$, in the sense that
$C\cap \La$ contains a uniformly perfect Cantor subset\footnote{That is, a Cantor set with no gaps that are disproportionately large at any scale.}, even though $C\cap \La$ itself need not be uniformly perfect.
Morally speaking, $\cal R$ consists of all frames connecting these uniformly perfect Cantor subsets of circular slices, modulo certain horoballs. 
The link between thick circular slices and thick recurrence of unipotent flows stems from the fact that the visual images of horocycles on $\hc$ are circles.

 \begin{figure}[htbp] \begin{center}
   \includegraphics [height=3.5cm] {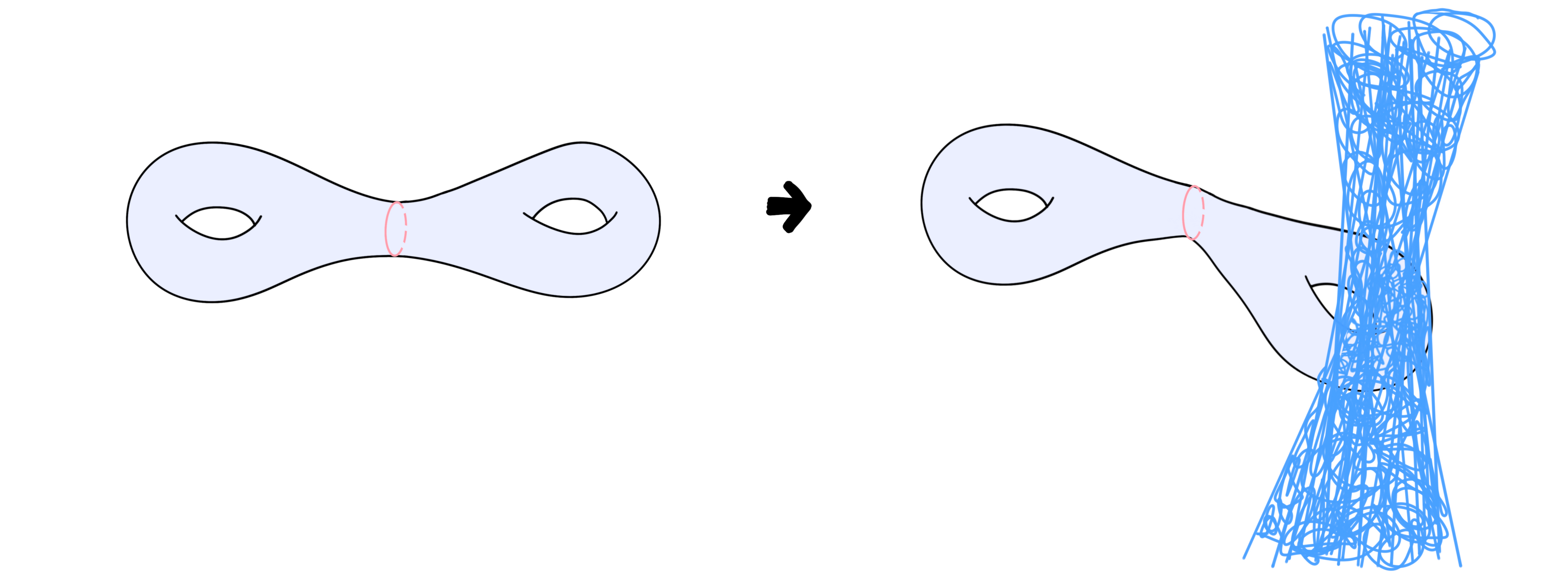}  
   \end{center}
     \caption{Bending and chaotic geodesic plane}\label{bend}
  \end{figure} 

\subsection{Examples of chaotic geodesic planes in quasi-Fuchsian manifolds}\label{s:chaos}  Indeed, the rigidity of geodesic planes does not hold for general geometrically finite hyperbolic $3$-manifolds when there are circular slices of $\La$ that fail to be thick enough.
A Kleinian group $\Gamma$ is called {\it quasi-Fuchsian} if it is a quasiconformal deformation of a (Fuchsian) lattice of $\PSL_2(\br)$; its limit set $\Lambda$ is a Jordan circle (see
the first image in Figure \ref{limitex}).
In \cite{MMO_inventiones}, it was shown that many quasi-Fuchsian $3$-manifolds obtained via
 bending constructions contain geodesic planes whose closures are diffeomorphic to the closure of a geodesic in a closed hyperbolic surface times the real line. Because geodesic closures in closed hyperbolic surfaces are known to be chaotic, the corresponding geodesic planes in these hyperbolic $3$-manifolds likewise exhibit chaotic behavior (see Figure \ref{bend}).

\section{Orbit closures in higher dimensions}
\subsection{Ratner's theorem in finite volume}
In this section, we consider higher dimensional analogues of convex cocompact rigid acylindrical hyperbolic $3$-manifolds and present rigidity results for geodesic planes and horocycles.
We begin by recalling Ratner's theorem on orbit closures of subgroups generated by unipotent elements, which was a conjecture of Raghunathan.
\begin{theorem}[\cite{ratner_top}]\label{ra}
    Let $G$ be a connected linear Lie group and $\Ga<G$ a lattice.
    Let $U<G$ be a connected closed subgroup generated by unipotent elements. Then for any $x\in \Ga\ba G$, the closure of $xU$ is homogeneous:
    $$\overline{xU} =xL $$
    for some connected Lie subgroup $L<G$ containing $U$.
\end{theorem}
For $G=\PSL_2(\c)$, the subgroup $\PSL_2(\br)$ is a maximal connected Lie subgroup of $G$, and hence Ratner's theorem yields a closed or dense dichotomy for $\PSL_2(\br)$-orbits in $\Ga\ba \PSL_2(\c)$ when $\Ga$ is a lattice. This special case was also proved independently by Shah \cite{Shah_1991}, following earlier ideas of Margulis \cite{Margulis_Banach}, and Dani-Margulis \cite{Dani_Margulis1989} in their work on the Oppenheim conjecture on values of quadratic forms.

Classifications become more subtle as the acting subgroup $U$ becomes smaller, since more orbit-closure possibilities can arise. The most delicate and essential case of Ratner’s theorem is when $U$ is a one-parameter unipotent subgroup.
For geometrically finite rigid acylindrical hyperbolic $3$-manifolds, a classification of orbit closures for a one-parameter unipotent subgroup $U<\PSL_2(\mathbb C)$
was obtained in \cite{MMO_gafa} for convex cocompact groups and extended in \cite{kimlee} to geometrically finite groups. In dimension two, the analogous classification had already been established much earlier by Hedlund \cite{Hedlund_duke} and by Dal'bo \cite{dalbo}.

\subsection{Ratner-type rigidity in hyperbolic manifolds with Fuchsian ends} In joint work with Lee \cite{LeeOh_orbit}, we considered convex cocompact hyperbolic $d$-manifolds with Fuchsian ends for any $d\ge 4$ and proved an analogue of Ratner's theorem on their frame bundles.
Let $\bH^d$ denote the $d$-dimensional hyperbolic space with boundary $\partial \bH^d=\bS^{d-1}$, and let
$$G:=\SO(d,1)^\circ=\op{Isom}^+(\bH^d).$$
Any complete hyperbolic $d$-manifold can be written as $\cM=\Gamma\ba \bH^d$ for a torsion-free discrete subgroup $\Gamma <G$.
The limit set $\La$ of $\Gamma$ and the convex core of $\cM$ are defined exactly as in the three-dimensional case.  

 \begin{definition}
     A convex cocompact hyperbolic $d$-manifold $\cM$ is said to have {\it Fuchsian ends}
 if the boundary of its convex core is totally geodesic.
  \end{definition}
See Figure \ref{limit} for an illustration of the limit set in this setting.
\begin{figure}[htbp]   \label{limit} \begin{center}
 \includegraphics [height=4cm]{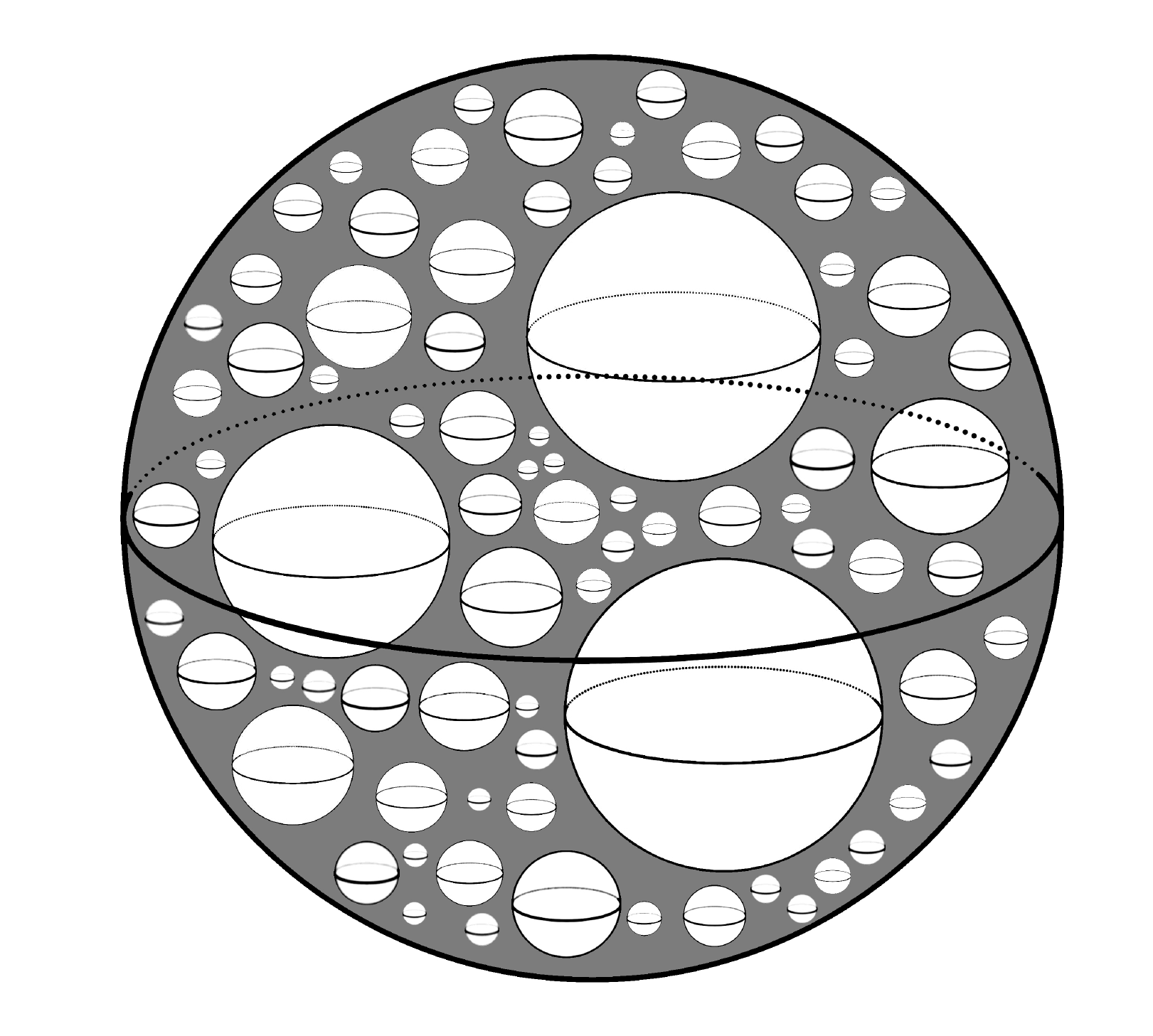}
\end{center} \caption{Limit set of a convex cocompact hyperbolic $4$-manifold with Fuchsian ends}
 \end{figure} 
 
The term {\it Fuchsian ends} reflects the fact that each component of the complement $\cM-\core (\cM)$ is diffeomorphic to the product $S\times (0,\infty)$ for some compact hyperbolic $(d-1)$-manifold $S$.  For $d=2$, every convex cocompact hyperbolic surface has Fuchsian ends. For $d=3$, these are precisely convex cocompact rigid acylindrical hyperbolic $3$-manifolds.
For $d\ge 4$, such manifolds are infinitesimally rigid, in contrast to the cases $d=2,3$: the inclusion $\Gamma\to G$ admits no non-trivial local deformations \cite{Kerckhoff_Storm}.

Let $A$ be the one parameter diagonalizable subgroup of $G$ generating the frame flow. The renormalized frame bundle of $\cM$ is the set of all  frames whose $A$-orbit connects a pair
of limit points; equivalently,
$$\RFM={\{x\in \Gamma\ba G: \text{$xA$ is bounded}\} }.$$
This set captures all the non-trivial dynamics of $A$-action on $\Ga\ba G$. The following joint work with Lee provides an analogue of Ratner's theorem:
 \begin{theorem}[\cite{LeeOh_orbit}]\label{higherratner}Let $\cM$ be a convex cocompact hyperbolic $d$-manifold with Fuchsian ends with $d\ge 2$.
 Let $U<G$ be any connected closed subgroup generated by unipotent elements, normalized by $A$.
Then for any $x\in \RFM$, the closure of $xU$ is relatively homogeneous: 
\be\label{oo} \overline{xU}\cap \RFM =xL\cap \RFM \ee 
for  a connected closed subgroup $U<L<G$ such that $xL$ is closed.
\end{theorem}
From this, one can deduce the classification of
the entire orbit closure $\overline{xU}$, not just its intersection with $\RFM$:
indeed, 
$$\overline{xU}=xL\cap \overline{(\RFM ) U},$$
and the structure of $\overline{(\RFM ) U}$ is well-understood. 

To highlight a new ingredient in the proof of Theorem \ref{higherratner} in higher dimensions, consider a one-parameter unipotent subgroup $U$, and suppose that $xU$ is not contained in any closed orbit $xL$ of a proper intermediate subgroup $U<L<G$. In this situation, one must show that ${xU}\cap \RFM$ is dense in $\RFM$. The obstruction arises because the thick recurrence of $xU$ to $\RFM$ may keep the orbit too close to the singular set, which is the union of all intermediate closed orbits. The key technical tool is an infinite-volume analogue of the {\emph{Dani-Margulis avoidance principle}} for the singular set (\cite{DM_soviet}), combined with a delicate inductive scheme that incorporates both the dimension of the maximal unipotent subgroup of the acting group and certain equidistribution statements.
Roughly speaking, the avoidance principle in this setting asserts that unipotent flows return to $\RFM$ in such a way that, during thick recurrence times, they remain at a definite distance from every compact subset of
the singular set (see \cite{LeeOh_orbit}, \cite{OhSurvey}).

\subsection{Rigidity of horocycles} As a geometric consequence of the orbit-closure classification for one dimensional unipotent subgroups, we obtain the rigidity of horocycles. A horocycle in $\cM=\Ga\ba \bH^d$ is an isometrically immersed copy of $\br$ with zero torsion and constant geodesic curvature equal to $1$. In the upper half-space model
 $$\bH^d=\{(x_1, \cdots, x_{d-1}, y): y>0\},$$
 a horocycle is either a horizontal line or a circle tangent to the boundary $y=0$, and its image in $\cal M=\Ga\ba \bH^d$ gives rise to a horocycle  in $\cM$ via the quotient map $\bH^d\to \cM$.

\begin{corollary} [\cite{LeeOh_orbit}] \label{h2} Let $\cM$ be a convex cocompact hyperbolic $d$-manifold with Fuchsian ends with $d\ge 2$.
For any horocycle $\chi \subset \cM$,
 one of the following holds:
\begin{enumerate}
\item $\chi=\overline \chi$ is properly immersed; or
\item
$\overline{\chi}$ is
  a properly immersed totally geodesic $k$-plane for some $k\ge 2$, up to tilting.
\end{enumerate}
\end{corollary}
An immersed totally geodesic $k$-plane of $\cM$ is the image of a totally geodesic immersion $\bH^k\to \cM$.  The phrase “up to tilting” means the following: a tilted copy of $\mathbb H^k \subset \mathbb H^n$ is obtained by pushing $\mathbb H^k$ a fixed distance along its normal direction inside a $(k+1)$-dimensional hyperbolic subspace.  Tilting preserves the essential geometric features relevant to our setting: horocycles are carried to horocycles, along with their closures. It is therefore natural that Corollary \ref{h2} allows such
tilting operations.

For a more expository perspective, we refer the reader to \cite{Oh_traveler}, where the closure of a horocycle is portrayed as the journey of a traveler through the hyperbolic world, and the rigidity theorem is compared to Kronecker’s classical theorem on the closure of a line in a torus. See Figure \ref{traveler} for an illustration of this viewpoint.

\begin{figure}[htbp] \label{traveler} \begin{center}
 \includegraphics [height=5cm]{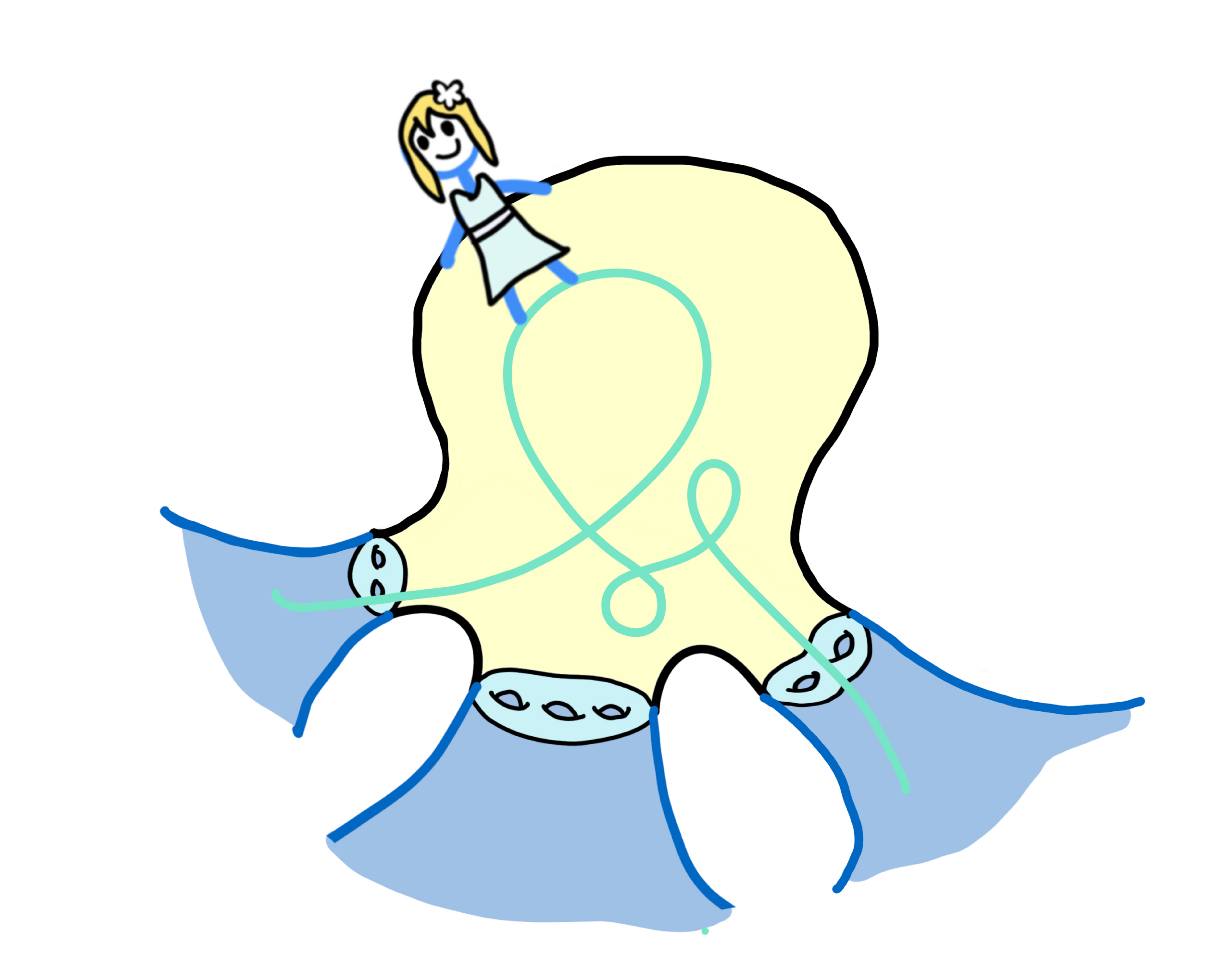}
\end{center} \caption{Sightseeing of a horocycle traveler in a hyperbolic manifold with Fuchsian ends}
 \end{figure}

\subsection{Rigidity of geodesic planes} 
The next result shows that closure rigidity extends beyond geodesic two-planes to all totally geodesic submanifolds of dimension at least two. 

\begin{corollary}[\cite{LeeOh_orbit}] \label{h1}
Let $\cM$ be a convex cocompact hyperbolic $d$-manifold with Fuchsian ends with $d\ge 3$.
Then the closure of any totally geodesic  submanifold of dimension at least $2$ intersecting $\op{core} \cM$ is a properly immersed totally geodesic submanifold.
\end{corollary}

\begin{remark} It seems unlikely that there exists an analogous class
of locally symmetric manifolds in higher rank where geodesic planes exhibit rigidity. Nevertheless, explicit higher-rank counterexamples to geodesic plane rigidity have been
 constructed only recently, in joint work with Dey \cite{Dey_Oh}, using floating geodesic planes and bulging deformations of Goldman. \end{remark}

\subsection{Orbit closure of circles in higher dimensions} Finally, to connect back with the orbit-closure problem for circles from the previous section, we include the following description of $\Gamma$-orbits of circles. As before, let $\mathcal C$ denote the space of circles in the boundary $\mathbb S^{d-1}$.

\begin{corollary}[\cite{LeeOh_orbit}] \label{sphere} 
Let $C\in \cal C$ be a circle intersecting $\La$ in more than two points.
Then there exists a $k$-dimensional sphere $S$ for some $k\ge 1$
such that
$$\cl{\Ga C}=\{D\in \cal C_\Lambda : D \subset  \Gamma S\}.$$
\end{corollary}
When $d=3$, the only possible spheres are circles or $\hc=\S^2$, recovering the dichotomy described earlier. In higher dimensions, circle orbits can accumulate on spheres of intermediate dimension, which reflects the richer geometric symmetries available beyond dimension three.

\section{Representation rigidity of Kleinian groups}\label{s:rig}
\subsection{Circular slices under quasiconformal conjugacy}
Figure \ref{f:def} below illustrates how the
limit set of a geometrically finite rigid acylindrical Kleinian group $\Gamma$ transforms
under a quasiconformal conjugacy $f:\hc\to \hc$.

  \begin{figure}[htbp] \begin{center}
   \includegraphics [height=5cm]  {cc_rigid} 
   \quad \includegraphics [height=5cm]  {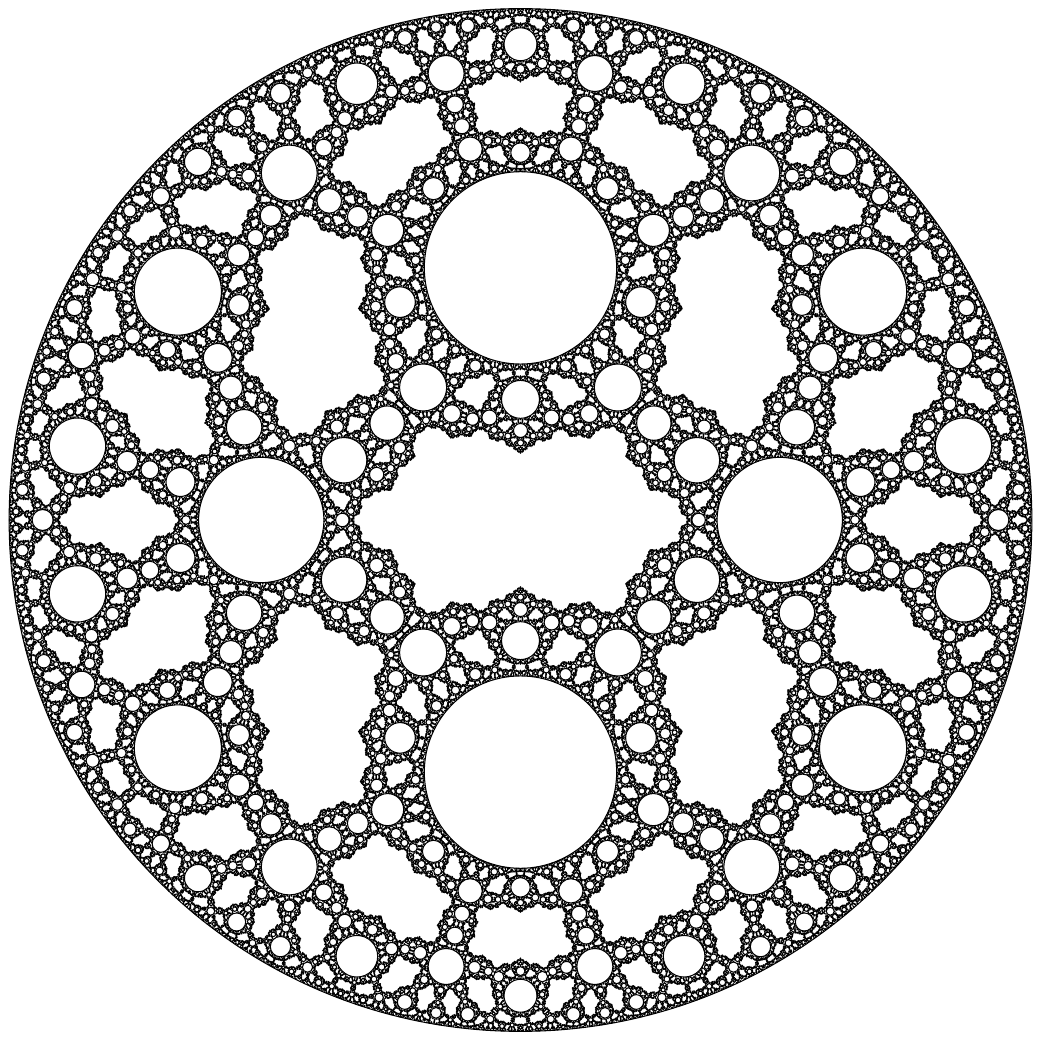}
   \end{center}
     \caption{Limit sets under quasiconformal deformation}\label{f:def}
  \end{figure} 
Clearly $f$ is not conformal: $f$ does not map all circles in $\hc$ to circles. It is also
not hard to see that $f$ cannot map any open collection of circles to circles.
On the other hand, it is much less clear how many circular slices of the limit set $\La$ can be mapped into circles under $f$. 
Denote by $\La_f$ the collection of all such circular slices:
\be\label{lf}
\Lambda_f:= \bigcup \left\{ C \cap \La : \begin{matrix}
C \subset \hc \mbox{ is a circle such that }
f(C \cap \La) \mbox{ is contained in a circle}
\end{matrix}
\right\}.
\ee 

We call a point $\xi\in \La_f$ a conformal point of $f$; that is, $\xi\in \La$ lies on some circle $C$ such that $f(C\cap \La)$ is contained in a circle.
Our results show that these conformal points $\La_f$ are negligible, both topologically and measurably: $\La_f$ has empty interior in $\Lambda$, and, even more, it has zero $\delta$-dimensional Hausdorff measure, where $\delta=\dim\Lambda$. This phenomenon is not isolated: it reflects a deeper principle of representation rigidity for Kleinian groups. 

\subsection{Mostow-Prasad and Sullivan rigidity}
To formulate this principle precisely, we recall the rigidity theorems of Mostow-Prasad and Sullivan.
Let $\Gamma<\PSL_2(\c)$ be a Kleinian group and consider the discrete and faithful locus of its representation variety:
$$\mathfrak R_{\op{disc}} (\Gamma):=\{\rho:\Ga \to \PSL_2(\c): \text{$\rho$ is discrete and faithful}\}.$$
For any $g\in \Mob(\hc)$, the conjugation $\ga\mapsto g \ga g^{-1}$ defines an element of
$\mathfrak R_{\op{disc}} (\Gamma)$. Such representations are precisely those induced by automorphisms of $\PSL_2(\c)$; we call them {\emph{algebraic representations}}. The Mostow-Prasad rigidity theorem\footnote{Mostow first proved this for cocompact lattices, and Prasad extended it to non-cocompact lattices.}  states that lattices admit no other discrete faithful representations:
\begin{theorem}[\cite{Mostowbook}, \cite{Prasad1973strong}] If $\Gamma<\PSL_2(\c)$ is a lattice, then every discrete faithful representation of $\Gamma$ is algebraic.
\end{theorem}

We are therefore led to the infinite volume case. For finitely generated groups, there are two types of the limit set  by the Ahlfors measure conjecture, now a theorem proved through the works of Thurston, Canary, Agol, and Calegari-Gabai on the tameness conjecture:
\begin{theorem} [Ahlfors measure conjecture] Let $\Ga<\PSL_2(\c) $ be finitely generated.
Then $$\La=\hc\quad\text{or}\quad \op{Leb}(\La)=0 ,$$
where $\op{Leb}$ denotes the Lebesgue measure of $\c$.
\end{theorem}

For finitely generated groups $\Ga$ with $\La=\hc$, Sullivan established the quasiconformal rigidity of $\Ga$ \cite{Sullivan1981ergodic}.
Define the Teichm\"uller space of $\Ga$ as the space of quasiconformal deformations:
$$\mathfrak T (\Gamma)=\{\rho\in \rdi:\text{there exists a $\rho$-equivariant quasiconformal homeomorphism of $\hc$}\}.$$
Quasiconformal deformations played an essential role in Mostow's proof of rigidity, which proceeds in two steps. For any lattice $\Ga<\PSL_2(\c)$, 
 \begin{enumerate}
     \item every $\rho\in \rdi$ yields an equivariant quasiconformal homeomorphism of $\hc$, so $\rdi=\mathfrak T(\Ga)$; 
\item  every such homeomorphism is in fact M\"obius, so
$\mathfrak T(\Ga)=\{\text{algebraic representations}\}$.
 \end{enumerate} 

 Sullivan's quasiconformal rigidity theorem states that the second step persists even in the infinite-volume case, provided that $\La=\hc$:
\begin{theorem}[\cite{Sullivan1981ergodic}]\label{su}If $\Ga$ is finitely generated with $\La=\hc$,
then any quasiconformal deformation of $\Ga$ is algebraic.
\end{theorem}
In fact, Sullivan proved that for any finitely generated $\Ga$ and
an equivariant quasiconformal homeomorphism $f$ of $\hc$, 
\be\label{sa} \text{$f$ is conformal on $\Omega$ $\Rightarrow$ $f$ is M\"obius} .\ee

Although this was not known at the time of Sullivan’s work, it now follows from the Ahlfors measure conjecture and the measurable Riemann mapping theorem that if $\Omega\neq\emptyset$, then \eqref{sa} is automatic. Thus its real content lies in the case $\Lambda=\widehat{\mathbb C}$, where it is precisely Sullivan’s quasiconformal rigidity.

\subsection{Quasiconformal rigidity via conformality on the limit set} We now turn to the case when $\La\ne \hc$. In this situation, by the Ahlfors finiteness theorem, $\Gamma\ba \Omega$ is a finite-type Riemann surface, and the Teichm\"uller space $\mathfrak T(\Ga)$ modulo conjugation is as large as the Teichm\"uller space
of $\Gamma\ba \Omega$ (\cite{Marden2016hyperbolic}, \cite{Matsuzaki1998hyperbolic}).
Hence one cannot expect an analogue of Theorem \ref{su}. It is therefore natural to ask for a criterion
ensuring that $\rho\in \mathfrak T (\Ga)$ is algebraic. Taking a cue from Sullivan's theorem \eqref{sa}, we are led to ask whether 
$$\text{$f$ is "conformal on $\La$" $\Rightarrow$ $f$ is M\"obius} .$$
When $\op{Leb} (\Lambda) =0$, the analytic notion of conformality does not make sense on $\Lambda$. 
The following joint result with Kim introduces a natural geometric substitute.
\begin{theorem}[\cite{KO_rigidity}] \label{koto} Let $\Gamma$ be a Kleinian group such that $\Omega$ has at least two components. Let $f$ be a $\Ga$-equivariant quasiconformal homeomorphism of $\hc$.
Suppose that neither $\La$ nor $f(\Lambda)$ is contained in a circle. 
If $f$ is conformal on $\Lambda$, in the sense that
$$\text{for any circle $C\subset \c$, $f(C\cap \La)$ is contained in a circle},$$ then
 $f$ is a M\"obius transformation.
\end{theorem}
\begin{figure} [htbp] \begin{center}
\includegraphics [height=4cm]{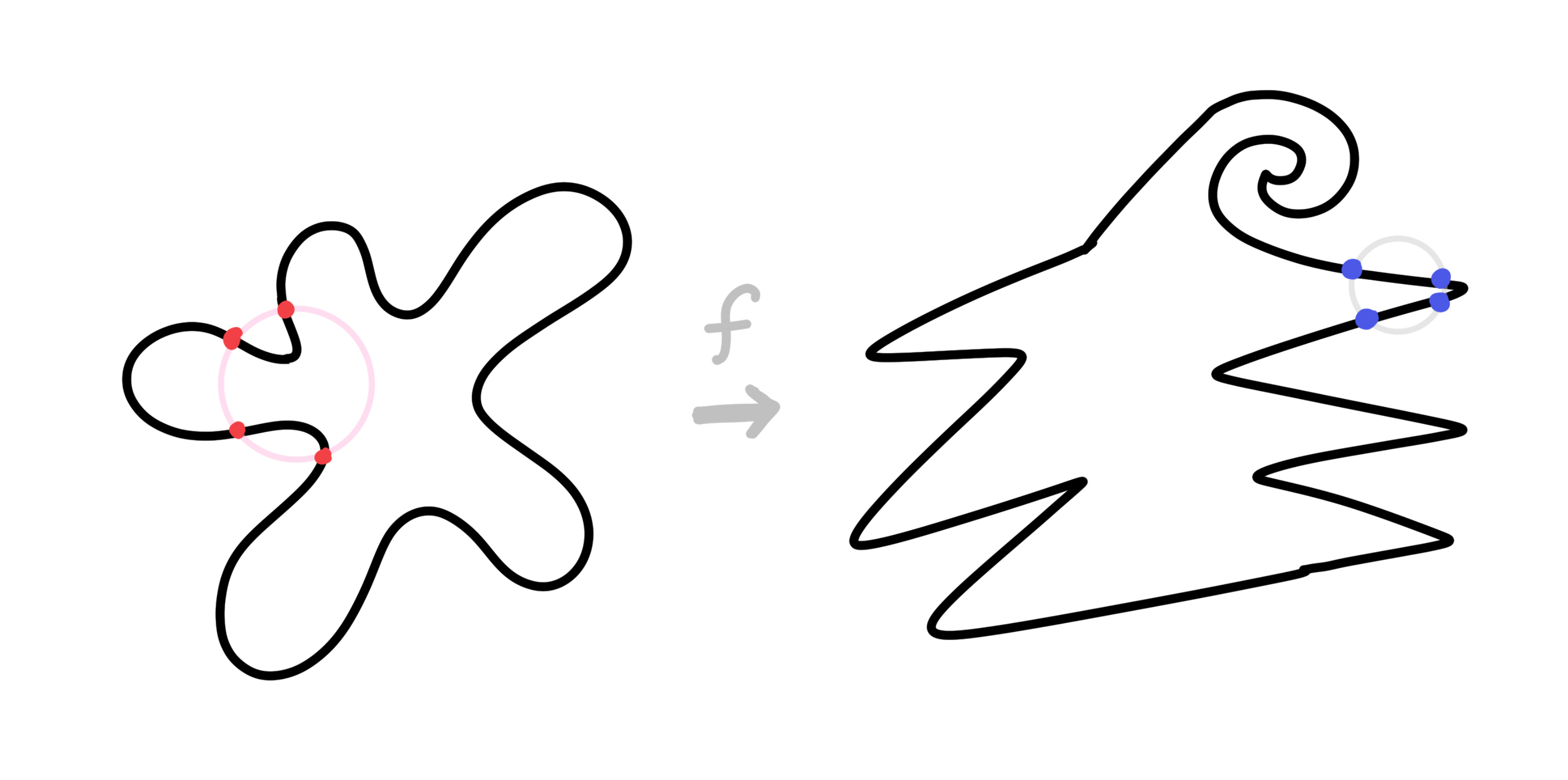}  \end{center}\caption{Conformal points for $f$}\label{conf}
\end{figure}  
The theorem requires only that $f$ maps circular slices of $\La$ into circles, not that it sends entire circles to circles (see Figure \ref{conf}).
Sullivan’s theorem \eqref{sa} covers the case $\Lambda=\hc$, while Theorem \ref{koto} provides the natural analogue when $\Lambda\ne \hc$. Taken together, they show that  any quasiconformal conjugacy that is conformal on the “visible part’’ of the dynamics must be M\"obius. The theorem applies broadly, for example, to all geometrically finite groups with connected limit set, which include all acylindrical and quasi-Fuchsian groups, but not Schottky groups.

\subsection{Rigidity via boundary maps on the limit set} We now present a stronger version that applies to any discrete faithful representation $\rho\in \rdi$ admitting a boundary map. By a boundary map, we mean a $\rho$-equivariant continuous embedding
$f: \La\to \hc$; it is unique when it exists.
If both $\Ga$ and $\rho(\Ga)$ are geometrically finite and $\rho$ is
type preserving\footnote{That is, it maps loxodromic (resp., parabolic) elements to loxodromic (resp., parabolic) elements.}, then $\rho$ admits a boundary map, as shown by Tukia. Unlike the boundary map of a quasiconformal deformation, which extends to all of $\hc$, this map is defined only on $\La$ in general.
See Figure \ref{boundarymap} for a schematic illustration of such a boundary map on the limit set. 
\begin{figure} [htbp]\label{boundarymap}
\begin{center}
  \includegraphics [height=4cm]{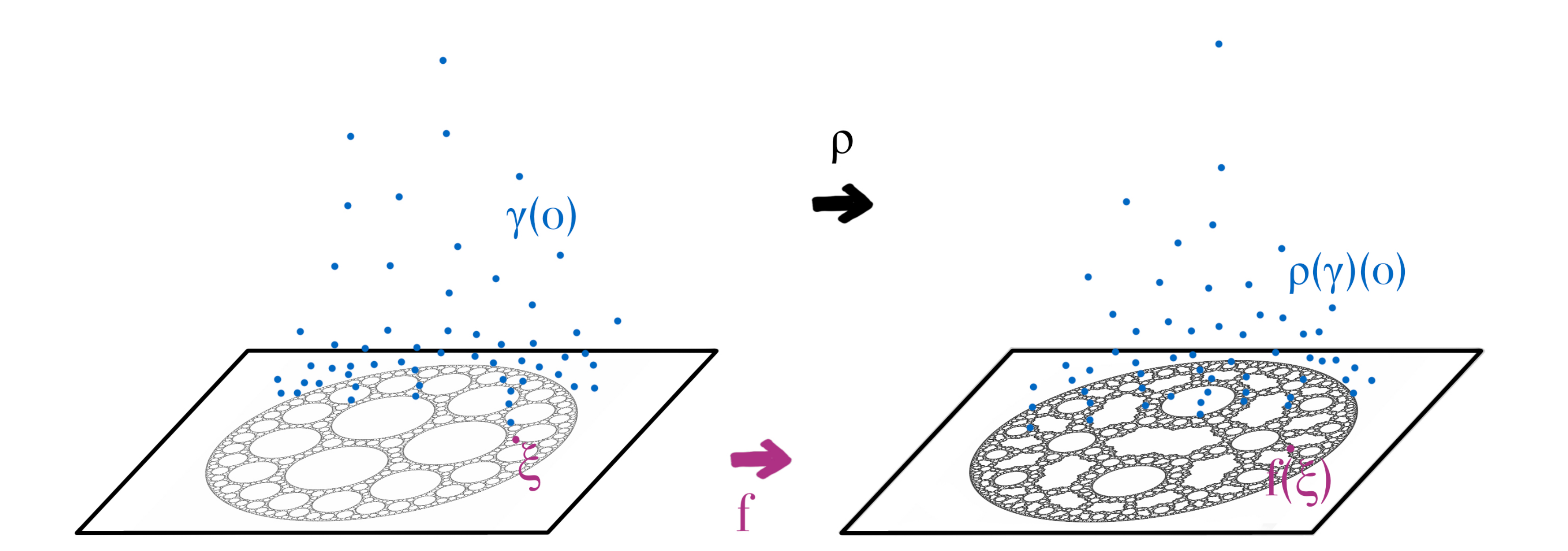}
 \end{center}\caption{The boundary map}
\end{figure}

In joint work with Kim, we obtained the following topological rigidity criterion, formulated in terms of the union $\La_f$ of all circular slices of $\La$ that are mapped to circles by $f$ (see \eqref{lf}).
\begin{theorem}[Topological criterion \cite{KO_rigidity}]\label{ko2a}Let $\Ga$ be a Kleinian group such that $\Omega$ has at least two components. Let $\rho\in \rdi$ and $f:\La\to \hc$ be its boundary map. Suppose that neither $\La$ nor $f(\Lambda)$ is contained in a circle. 
 Then
$$\text{either 
$\quad \La_f=\La\quad $ or $\quad \La_f\text{ has empty interior in $\La$}$}.$$
In the first case, $f$ extends to a M\"obius transformation and $\rho$ is algebraic.
\end{theorem}

As an immediate consequence, if $\Gamma$ as above admits a non-algebraic representation $\rho$ and  neither $\Gamma$ nor $\rho(\Gamma)$ is Fuchsian, then the limit set $\Lambda$ contains no open subset that is covered by circles intersecting $\La$ in at most three points.

Since the cross-ratio of four points of $\hc$ is real precisely when the four points lie on a circle,
this yields a rigidity criterion expressed purely in algebraic terms.
\begin{corollary} [Cross-ratio rigidity \cite{KO_rigidity}]
If $f$ sends every quadruple of points in $\Lambda$ with real cross-ratio to another quadruple with real cross-ratio, then $f$ extends to a M\"obius transformation. 
\end{corollary}
Inspired by the Ahlfors measure conjecture, we also obtain a measure-theoretic criterion:
\begin{theorem}[Measure theoretic criterion \cite{KO_JTo}]\label{ko2b}Suppose in addition that $\Ga$ is geometrically finite and $\rho$ is a type-preserving geometrically finite representation.
 Then
$$\text{ either 
$\quad \La_f=\La\quad $ or $\quad \nu (\La_f)=0$}$$
where $\nu$ is the unique $\Ga$-geometric measure on $\La$ (see Theorem \ref{psm}).
In the first case, $f$ extends to a M\"obius transformation and $\rho$ is algebraic.
\end{theorem}

In short, preservation of circular structure on the limit set acts as a rigidity principle: once it holds on a large enough set, the representation must be algebraic.

\subsection{Diagonal dynamics on the self-joining quotients}
These rigidity theorems are proved using the ergodic theory of diagonal flows on quotients of $\PSL_2(\c)\times \PSL_2(\c)$ by discrete subgroups called {\em self-joining groups}.  
For $\rho\in \rdi$, define the self-joining
\be \Ga_\rho=(\text{id}\times \rho)(\Ga)=\{(\ga, \rho(\ga)): \ga\in \Ga\} <G.\ee 
\begin{figure} [htbp]\label{selfjoining}
\begin{center}
  \includegraphics [height=5.5cm]{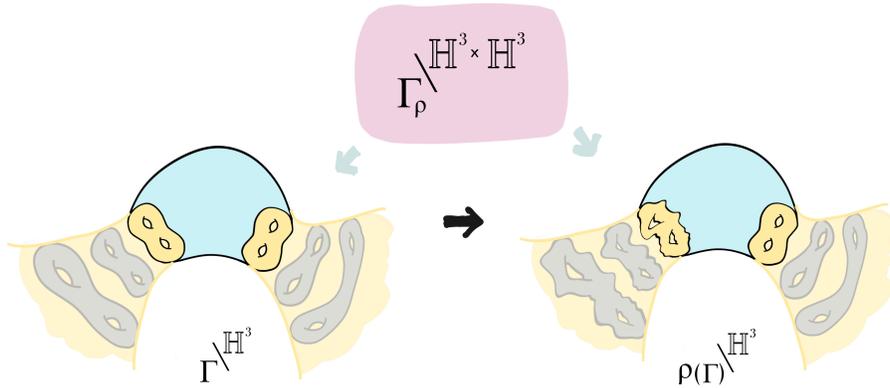}
 \end{center}\caption{Self-joining of two hyperbolic $3$-manifolds}
\end{figure}  
Figure \ref{selfjoining} depicts the diagonal embedding defining the self-joining quotient.
A basic but key observation is that
\be\label{sj} \text{$\rho$ is algebraic if and only if $\Ga_\rho$ is not Zariski dense in $\PSL_2(\c)\times \PSL_2(\c)$}.\ee 
Our rigidity theorems are obtained by showing that the conformality of $f$ on $\La$ obstructs the Zariski density of $\Ga_\rho$. 
While Mostow's rigidity theorem relies on the ergodicity of diagonal flows on $\Ga\ba \PSL_2(\c)$ with respect to the Haar measure, our approach uses the ergodic theory of diagonal flows on the higher-rank quotient space $\Gamma_\rho\backslash\left(\PSL_2(\mathbb{C})\times \PSL_2({\mathbb C})\right) $. The Furstenberg boundary of $\bH^3\times \bH^3$ is now $\hc\times \hc$ and the limit set of $\Ga_\rho$ in $\hc\times \hc$
is given by
$$\La_\rho = \{(\xi, f(\xi)): \xi\in \La\} .$$
Define $$R_\rho=\{ [(g_1, g_2)]\in \Ga_\rho\ba G: (g_1^{\pm}, g_2^{\pm}) \in \La_\rho\}$$
where $g_i^{\pm}\in \hc$ are visual images of $g_i$. 
One may view $R_\rho$ as the renormalized frame bundle of the higher-rank manifold $\Ga_\rho \ba (\bH^3\times \bH^3)$.
For each non-zero vector $\u=(u_1, u_2)\in \br^2$, the one-parameter diagonal subgroup 
$$A_{\u} =\left\{a_{t {\u}}=\left(\begin{pmatrix} e^{u_1t/2} & 0 \\ 0 &e^{-u_1t/2}\end{pmatrix},\begin{pmatrix} e^{u_2t/2} & 0 \\ 0 &e^{-u_2t/2}\end{pmatrix}\right):t\in \br \right\} $$ 
acts by right translations on $R_\rho$.
Unlike in rank one, where there is a single diagonal flow, the dynamics of $A_{\u}$ in rank two
depend crucially on the direction $\u$. The proof of Theorem \ref{ko2a} relies on the fact that
if $\Ga_\rho$ is Zariski dense, then for some $\u=(u_1, u_2)$ with $u_1, u_2>0$,
the $A_{\u}$-action on $R_\rho$ is topologically transitive,  i.e., it admits a dense orbit (\cite{Dang_top}, \cite{CS_local}).
The proof of Theorem \ref{ko2b} for $\Ga$ further uses the ergodicity of this action with respect to the higher rank Bowen-Margulis-Sullivan measure $m^{\BMS}_{\u}$ on $R_\rho$ \cite{BLLO}, which will be discussed further in section \ref{s:lmhigher}.

\subsection{A higher-rank perspective on Mostow rigidity}
The self-joining perspective also suggests a new approach to Mostow’s rigidity theorem: if $\Gamma<\PSL_2(\c)$ is a lattice, then for every $\rho\in \rdi$, the group $\Gamma_\rho$ cannot be Zariski dense in $\PSL_2(\c)\times \PSL_2(\c)$. Let $f:\hc\to \hc$ be the $\rho$-boundary map, which is a quasiconformal homeomorphism of $\hc$ with nonzero Jacobian at Lebesgue-almost every point.
Hence the Lebesgue measure on $\hc$ is absolutely continuous with respect to the push-forward $f_*\op{Leb}$. That this forces $\rho$
to be algebraic was already observed  by Sullivan \cite{Sullivan1982discrete} using the ergodicity of geodesic flow on $\Ga \ba\bH^3$. In joint work with Kim \cite{KO2023_conformalrigidity}, we provided a different proof using higher-rank Patterson-Sullivan theory for $\Ga_\rho$. More precisely, if $\Ga_\rho$ were Zariski dense, then ergodic properties of horospherical subgroup actions on $\Ga_\rho\ba G$ as in (\cite{LO_invariant}, \cite{KO2023_conformalrigidity}) would imply that the two pushforward measures 
$$(f^{-1}\times \op{id})_*\op{Leb}\quad\text{ and }\quad  (\op{id}\times f)_*\op{Leb}$$
on $\hc\times \hc$ -which arise as Patterson-Sullivan measures of $\Ga_\rho$, as in \eqref{higher}- must be mutually
singular. This, in turn, shows that $f_*\op{Leb}$ is singular to the Lebesgue measure,
contradicting absolute continuity. Therefore $\Ga_\rho$ cannot be Zariski dense, and hence $\rho$ is algebraic.

\section{From circle-counting to torus-counting} 
\subsection{Torus packings arising from quasiconformal deformations} 
In Theorem \ref{os2}, we presented a circle counting result for a locally finite circle packing $\cal P$ consisting of finitely many $\Ga$-orbits of circles. In this section, we present an analogous result for torus packing.
 By a torus in $\c^2$,  
we mean a pair of circles $T=(C_1, C_2)$. Let $\cal T$ denote the space of tori in $\c^2$. The volume of $T\in \cal T$ is defined as
$$\vol (T)=  \,\op{rad}( C_1) \times  \op{rad}( C_2).$$
Figure~\ref{torus} illustrates a quasiconformal deformation giving rise to a torus.
The main results from  Sections \ref{s:mmo} and \ref{s:rig} provide natural collections of torus packings to which
our counting theorem applies. 
\begin{theorem} \label{t1} Let $\Ga<\PSL_2(\c)$ be a convex cocompact rigid acylindrical group. Let 
$f$ be a $\rho$-equivariant quasiconformal homeomorphism of $\hc$ for some $\rho\in {\mathfrak T}(\Ga)$ that is not algebraic.
Define $$\cal P=\{(C, f(C))\in \cal T: \# (C\cap \La ) \ge 2 \}. $$
Then $\cal P$ is a finite union of closed $\Ga_\rho$-orbits of tori. In particular, $\cal P$ is locally finite\footnote{That is, for any bounded $B\subset\c^2$ and $\e>0$, there are only finitely many tori in $\cal P$ with volume at least $\e$ that intersect $B$.}.
\end{theorem} 
\begin{proof} Since $\La$ is a round Sierpi\'nski carpet, the condition
$\# C\cap \La \ge 2$ means either $C\subset \La$ or $C$ separates $\La$. Circles of the first type
form finitely many closed $\Ga$-orbits of circles by Theorem \ref{finite}. If $\Ga C$ is closed,
then circles in $\Gamma C$ accumulate only at radius $0$, and since
 $f$ is quasiconformal, the same holds for $f(C)$. 
Therefore it suffices consider
    $\cal P^*=\{(C, f(C))\in \cal T:  C\in \cal C_\La^* \}. $
    By Theorem \ref{c2}, the set $\{C\in \cal C_\La^*: f(C)\in \cal C\}$ is either dense
    or a finite union of closed $\Ga$-orbits. In the first case,
    $f$ maps a dense collection of circles in $\cal C_\La$ to circles, and hence
    by continuity, $f(\cal C_\La)\subset \cal C$. By Theorem \ref{koto}, $f$ would then be
    M\"obius, contradicting the hypothesis. Therefore the set $\{C\in \cal C_\La^*: f(C)\in \cal C\}$ is a finite union of 
     closed $\Ga$-orbits. By the $\rho$-equivariance of $f$, this implies the claim.
\end{proof}
\begin{figure}[h]
\begin{center}
      \includegraphics[totalheight=3.5cm]{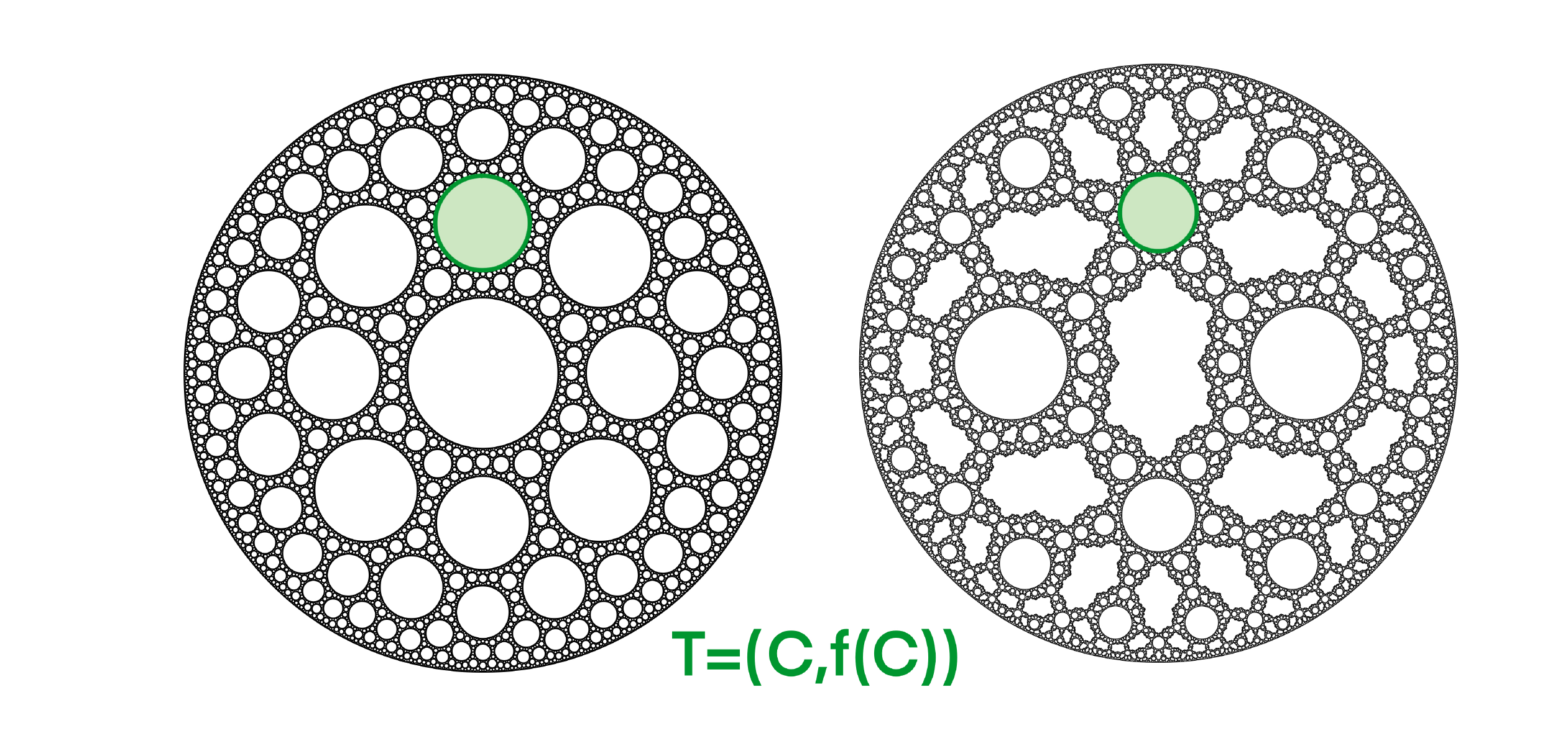}\end{center}\caption{Torus}\label{Torus}
\end{figure}
 In joint work with Edwards and Lee, we have proved
the following torus-counting theorem:

\begin{theorem}[\cite{ELO_torus}]
Let $\P$ be the torus packing from Theorem \ref{t1}.    
Then there exist constants $c_{f}>0$ and $\delta_f>0$ such that as  $t\to \infty$,
we have 
$$ \#\{T\in \cal P: \Vol (T)> \tfrac{1}{t} \} \sim c_{f} \cdot t^{\delta_f} .$$
Moreover, there exists a locally finite Borel measure $\omega_f$ on the set
$\{(\xi, f(\xi))\in \c^2: \xi\in \La\}$ such that
for any bounded region $R\subset \c^2$ whose boundary is a piecewise algebraic subvariety, 
$$ \#\{T\in \cal P: \Vol (T)> \tfrac{1}{t}, \; T\cap R\ne \emptyset  \} \sim c_{f} \cdot \omega_f (R)  \cdot t^{\delta_f} .$$
\end{theorem}

\subsection{Higher dimensional torus packings via self-joinings} 
For $n\ge 1$, an $n$-dimensional torus packing in $\c^n$ is a countable collection of $n$-dimensional tori $T=(C_1, \cdots, C_n)$ with $C_i\in \cal C$. 
Let $\Ga<\PSL_2(\c)$ be a convex cocompact, Zariski dense Kleinian subgroup, and let
$\rho=(\rho_1, \cdots, \rho_n)$ be an $n$-tuple of faithful, convex cocompact, Zariski dense representations of $\Ga$ into $\PSL_2(\c)$ with $\rho_1=\text{id}$, and assume that no two $\rho_i$ are conjugate in $\Mob (\hc)$. The {\it self-joining} of $\Gamma$ via $\rho$ is the following discrete subgroup of $G=\prod_{i=1}^n \PSL_2(\c)$:
$$\Gamma_\rho =(\prod_i\rho_i) (\Ga)
=\{\big(\rho_1(\ga), \cdots , \rho_n(\ga)\big): \ga\in \Gamma\}.$$
Let  $f=(f_1, \cdots, f_n)$ where  $f_i:\La\to \La_{\rho_i(\Ga)}$ is the $\rho_i$-boundary map.
In the following torus-counting result, obtained jointly with Edwards and Lee, we consider torus packings arising from these boundary maps. We
do not assume that $f_i(C_1)$ is a circle, but only that $f_i(C_1\cap \La)$ is contained in a circle $C_i$. 
\begin{theorem}[\cite{ELO_torus}]\label{elo}Let $\P$ be a locally finite torus packing consisting of finitely many $\Ga_\rho$-orbits.
Suppose that $\P$ is $f$-admissible, meaning that for any $(C_1, \cdots, C_n)\in \cal P$, 
$$ f_i(C_1\cap \La)= C_i\cap f_i(\La)\quad \text{for all $1\le i\le n$.}$$
  Then there exist a constant $0<c_{\cal P}<\infty $ and a locally finite Borel measure $\omega_\rho$ on $(\prod_if_i) (\La)\cap \c^n$ such that
  for any region $R\subset \c^n$ enclosed by a piecewise real algebraic subvariety,
  $$\#\left\{T\in \cal P: \vol (T) \ge \tfrac{1}t,\; T\cap R\ne \emptyset  \right\} \sim c_{\cal P} \cdot \omega_\rho (R) \cdot  t^{\delta_{\rho, \Vol}}\quad\text{as $t\to \infty$}$$
  where $\Vol (T)= \prod_{i=1}^n  \op{rad} (C_i)$ and $$ 0<\delta_{\rho, \Vol}=\limsup_{t\to\infty}\frac{1}{t}\log \#\{\ga\in \Ga:  \sum_{i=1}^n
  d(\rho_i(\ga) o, o) \le t\}< \frac{2}{\sqrt n}, \quad o\in \bH^3. $$
\end{theorem}

\subsection{The exponents in torus counting} The appearance of the particular exponent $\delta_{\rho,\Vol}$ in the above theorem is genuinely a higher rank phenomenon.
In the circle counting theorem, there was only one relevant exponent, namely the the critical exponent of a Kleinian group $\Ga$
with respect to the hyperbolic metric. In higher rank, however, the exponent depends on how the tori are
ordered in the counting problem. In Theorem \ref{elo}, ordering the tori by their volume leads to the exponent $\delta_{\rho,\Vol}$, which differs from the Riemannian critical exponent of $\Ga_\rho$.
For instance, if we replace  $\vol (T)$ by
$\prod \text{rad}(C_i)^{\kappa_i}$ with $\kappa_i>0$,
then $\delta_{\rho, \Vol}$ is replaced by 
$\limsup_{t\to\infty}\frac{1}{t}\log \#\{\ga\in \Ga:  \sum_{i=1}^n
  \kappa_i d(\rho_i(\ga) o, o) \le t\}$ \cite{ELO_torus}.

\subsection{Anosov subgroups and local mixing} The convex cocompact hypotheses on  $\rho$ in Theorem \ref{elo} ensure that the self-joining $\Ga_\rho$ belongs to the important class of  Anosov subgroups of $G$. Roughly speaking, Anosov subgroups are higher-rank analogues of convex cocompact Kleinian groups, characterized by strong geometric and dynamical stability properties (see \ref{s:anosov} for a precise definition).
The local mixing  of diagonal action on the Anosov quotients
$\Ga_\rho\ba G$ is the main tool in the proof. Local mixing is a fundamental phenomenon in homogeneous dynamics in infinite volume and will be discussed in detail in the next two sections.
The analogue of Theorem \ref{elo} for self-joinings of general geometrically finite groups remains open, with the main obstacle being the absence of a local mixing theorem in this broader setting.

\section{Matrix coefficients and Mixing in rank one infinite volume}
 In the next two sections, we discuss mixing for diagonal subgroups in
quotients of $G=\prod_{i=1}^n \SO(d,1)^\circ$ with $d\ge 2$ and $n\ge 1$.
The case of $n=1$, treated in this section,
corresponds to the frame flow on a hyperbolic $d$-manifold. Our discussion proceeds in three steps:
\begin{itemize}
    \item {\em local mixing} which captures renormalized correlation in infinite volume;
    \item {\em exponential mixing} which provides exponential rates of decay;
    and 
    \item {\em uniform exponential mixing} over congruence covers, with number-theoretic consequences including applications to the affine sieve.
\end{itemize}

\subsection{What is local mixing?} Let $G$ be a connected semisimple real algebraic group and $\Ga<G$ be a discrete subgroup.  Let $dx$ denote a $G$-invariant measure on $\Ga\ba G$ and write $\op{Vol}(S)$ for the volume of a subset $S\subset \Ga\ba G$ with respect to $dx$. Let $\{a_t:t\in \br\}$ be a one-parameter diagonalizable subgroup of $G$.
In simple terms, mixing concerns the distribution of $\cal O a_t$ as $t\to +\infty$ for a bounded open subset $\cal O\subset \Ga\ba G$.
When $\op{Vol}(\Ga\ba G)=\infty$, the Howe-Moore theorem on decay of matrix coefficients \cite{Howe_Moore} implies that for any bounded subset $B\subset \Ga\ba G$,
$$\lim_{t\to +\infty} \op{Vol} (\cal Oa_{t}\cap B) =0 ,$$
which by itself yields no meaningful distributional information.
The natural question is whether one can renormalize this quantity to obtain a nontrivial measure describing the asymptotic distribution of $\cal O a_t$. More precisely, one seeks a locally finite Borel measure $\mu$ on $\Ga\ba G$ such that for any bounded subsets $B_1, B_2\subset \Ga\ba G$ with boundaries of volume zero,
\be\label{oeq2} \lim_{t\to +\infty} \frac{\op{Vol}(\cal O a_{t} \cap B_1)}{\op{Vol}(\cal Oa_{t} \cap B_2)}= \frac{{\mu}( B_1)}{{\mu}( B_2)}\ee whenever $\mu(B_2)>0$.  This question can be formulated in terms of matrix coefficients of the quasi-regular representation of $G$ on 
 $L^2(\Ga \ba G, dx)$. 
For  $f_1, f_2\in C_c(\Ga\ba G)$, the corresponding matrix coefficient is the correlation function:
$$g\mapsto   \langle g. f_1,  f_2\rangle  =\int_{\Ga\ba G} f_1(x g ) f_2(x) \; dx .$$

\begin{definition} The action of $\{a_t\}$ on $ (\Ga\ba G, dx) $ is called {\it local mixing} if there exists a proper function $\Phi:\br_{\ge 0} \to \br_{\ge 0}$ and locally finite Borel measures $\mu_1,\mu_2$ on $\Ga \ba G $ such that for any $f_1, f_2\in C_c(\Ga\ba G)$,
\be\label{ma2} \lim_{t\to +\infty} \Phi(t) \langle a_t. f_1, f_2\rangle  =\int f_1 \,d\mu_1 \cdot \int f_2\, d\mu_2.
\ee  \end{definition}
When $\Phi$ is a constant function and $d\mu_1(x)=d\mu_2(x)=dx$, local mixing becomes the usual notion of strong mixing. The term {\it local} indicates that we  consider only compactly supported test functions so that their integrals with respect to any locally finite measure are well-defined. 
Note that \eqref{ma2} implies that \eqref{oeq2} holds with $\mu=\mu_1$. 

Roughly speaking, 
the renormalization function $\Phi(t)$, if exists, should satisfy that
for an open subset $\cal O\subset G$,
$$\Phi(t) \approx \frac{\op{Vol}( {\cal O} a_t {\cal O}) }{\# (\Ga \cap  {\cal O} a_t {\cal O})} \quad\text{ as $t\to +\infty$}; $$
this can be deduced by taking $f$ to be the characteristic function of $\Ga\ba \Ga \cal O$.
Therefore, the growth rate of $\Ga$ in the direction of the one-parameter subgroup $\{a_t\}$ determines $\Phi(t)$. The notion of local mixing was formally introduced in \cite{Oh_Pan_ZD}, where it was proved for the frame flow on abelian covers of convex cocompact hyperbolic manifolds.

\subsection{Local mixing of frame flow in geometrically finite hyperbolic manifolds}\label{s:so} Let $G=\SO(d,1)^\circ=\Isom^+(\bH^d)$ with $d\ge 2$, and $\Ga<G$ be a discrete subgroup. Write $\cM=\Ga\ba \bH^d$.

Let $A=\{a_t:t\in \br\}$ be a one-parameter diagonal subgroup of $G$, which is unique up to conjugation, and set $A^+=\{a_t:t\ge 0\}$. Fix a maximal compact subgroup $K<G$ so that the Cartan decomposition $G=KA^+K$ holds, and let $M$ denote the centralizer of $A$ in $K$.
Through the identifications $ \T^1(\cM)\simeq \Ga\ba G/M$ and $\op{F}(\cM)\simeq \Ga\ba G $, 
the right translation action of $A$ on $\Gamma\backslash G/M$ corresponds to the geodesic flow, while the action on $\Gamma\backslash G$ corresponds to the frame flow. The following local mixing theorem was proved by Roblin for the geodesic flow and by Winter for the frame flow.
\begin{theorem}[\cite{Roblin2003ergodicite}, \cite{Winter_mixing}]\label{win} Let $\Ga<G$ be a geometrically finite and Zariski dense subgroup.
For any $f_1, f_2\in C_c(\Ga\ba G)$,
    we have $$\lim_{t\to +\infty} e^{(d-1-\delta)t } \langle a_t. f_1,  f_2\rangle   =\frac{1}{|m^{\BMS}|} m^{\op{BR}}(f_1) \cdot m^{\op{BR_*}} (f_2),$$
    where $m^{\op{BR}}$ and  $m^{\BR_*}$ denote the stable and unstable Burger-Roblin measures, respectively.
\end{theorem}

\subsection{Burger-Roblin measures} \label{s:br}
Let $o\in \bH^d$ be a point with stabilizer $K$. Let $\nu_o$ denote the Patterson-Sullivan measure on $\La$ and $m_o$ the $K$-invariant measure on $\partial \bH^d=\bS^{d-1}$. As in dimension $3$, we use
the Hopf parameterization $\T^1(\bH^d)= \{v=(v^+, v^-, t=\beta_{v^+}(o, v)): v^{\pm}\in \S^{d-1}\}$. In these coordinates, the Burger-Roblin measure
$m^{\BR}$ is induced from the $\Ga$-invariant measure on $\T^1(\bH^d)$ given by
\begin{align*}
d \tilde m^{\BR}(v)& =
e^{\delta \beta_{v^+}(o, v)}\; e^{(d-1) \beta_{v^-}(o,v) }\;d\nu_o(v^+) dm_o(v^-) dt. \end{align*}
The key distinction from the Bowen-Margulis-Sullivan (BMS) measure is the use of $m_o$ at the backward endpoint $v^-$ instead of $\nu_o$.
This choice makes $m^{\BR}$  invariant under the stable horospherical subgroup
$N=\{g\in G: a_{-t} g a_{t}\to e\text{ as $t\to +\infty$}\} $, that is,
$dm^{\BR}(xn)= dm^{\BR}(x)$ for all $n\in N$.
In fact, $m^{\BR}$ is
the unique ergodic $N$-invariant measure on $\Ga\ba G$ not supported on a single $N$-orbit, up to scaling (\cite{burger_horo}, \cite{Roblin2003ergodicite}, \cite{Winter_mixing}). It is always infinite except when $\Ga$ is a lattice. This measure was introduced by Burger in \cite{burger_horo} and was further studied by Roblin \cite{Roblin2003ergodicite}.
The unstable Burger-Roblin measure $m^{\BR_*}$ is defined analogously, with the roles of $g^{\pm}$ interchanged, and is invariant under the unstable horospherical subgroup $\check N=\{g\in G: a_{t} g a_{-t}\to e\text{ as $t\to +\infty$}\}$. It turns out that Theorem \ref{win} is equivalent to the strong mixing of the finite BMS measure (\cite{Babillot_mixing}, \cite{Winter_mixing}) via Roblin's transversal intersection argument (see \cite{Roblin2003ergodicite}, \cite{OS_Jams}). On the other hand, local mixing with respect to the Haar measure is often more useful in equidistribution and counting theorems, since Haar measures are easier to handle under changes of variables.

\subsection{Exponential mixing in infinite volume}
By "exponential mixing", we mean exponential {\emph{local mixing}}, i.e.  exponential error terms in the renormalized correlation asymptotics of Theorem \ref{win}. The discussion splits into two cases, depending on the size of the critical exponent $\delta$, which reflects the $L^2$-spectrum of
the hyperbolic manifold $\cM=\Ga\ba \bH^d$. In fact,  geometrically finite Kleinian groups can have critical exponent arbitrarily close to $0$ or to $d-1$.
Identifying $L^2(\cM)$  with the subspace of $K$-invariant functions in $L^2(\Ga\ba G)$, we see that the asymptotics of its matrix coefficients are tied to the $L^2$-spectrum of the Laplacian.
Denote by $\Delta$ the negative of the Laplace operator on $\cM$. 
Combining results of Patterson, Sullivan and Lax-Phillips, one obtains the following theorem:
\begin{theorem}[\cite{Patterson1976limit}, \cite{Sullivan1979density}, \cite{Sullivan_Riemannian}, \cite{Lax_Ph}]  \label{lax} Let $\cM=\Ga\ba \bH^d$ be a geometrically finite hyperbolic manifold.
\begin{itemize}
    \item If $\delta >\frac{d-1}{2}$, then the bottom eigenvalue of $\Delta$ on $L^2(\cM)$ is simple and equal to $\lambda_0=\delta(d-1-\delta)$, and
 there are only finitely many eigenvalues of $\Delta$ in $[\lambda_0, \frac{(d-1)^2}{4})$. 
 \item If $\delta\le \frac{d-1}{2}$, then the $L^2$-spectrum of $\Delta$ is purely continuous.
 \end{itemize}
\end{theorem}

\subsubsection*{Case $\delta>\frac{d-1}{2}$:} In this regime, $L^2$-spectral theory and unitary representation theory can be used to establish exponential mixing. Let $\lambda_1= s_1(d-1-s_1) $ denote
the second smallest eigenvalue of $\Delta$
 in $[\lambda_0,  \frac{(d-1)^2}{4})$, with $s_1=\frac{d-1}{2}$ if there is no such eigenvalue, and 
define the spectral gap $$\eta_\Ga:=\min(\delta-s_1, 1).$$

The following result with Edwards shows that for the geodesic flow, the exponential rate of local mixing is determined precisely by the spectral gap $\eta_\Gamma$.
\begin{theorem} [\cite{Edwards_Oh_tem}]\label{eo}  
Let $\cM$ be geometrically finite with $\delta  >\frac{d-1}{2}$. There exists $m\in \N$ (depending only on $G$) such that,
 for any $\e>0$ and any $f_1, f_2\in C_c(\Ga\ba G)^M=C_c(\T^1(\cM))$ with  ${\mathcal S}^m(f_1), {\mathcal S}^m(f_2)<\infty$,
    \be \label{eta} \lim_{t\to +\infty} e^{(d-1-\delta)t } \langle a_t. f_1, f_2 \rangle  =\frac{1}{|m^{\BMS}|} m^{\op{BR}}(f_1) \cdot m^{\op{BR_*}}(f_2) + O_\e \left( e^{(-\eta_\Ga +\e )t }\mathcal S^m(f_1)\mathcal S^m(f_2) \right) \ee
    where  
 ${\mathcal S}^m(f_i)$ denotes the $L^2$-Sobolev norm of degree $m$.
\end{theorem}
 In earlier joint work with Mohammadi \cite{MohammadiOh_JEMS}, exponential mixing for the frame flow was established under the assumption $\delta > d-2$ with
an error term controlled by a parameter depending on $\eta_\Gamma$.
Theorem \ref{eo} sharpens this picture by showing that, for the geodesic flow, the rate is exactly $\eta_\Gamma$ itself.

\subsubsection*{Case $\delta\le \frac{d-1}{2}$:}
In this range, the
$L^2$-spectrum of $\Delta$ is purely continuous, and spectral theory alone does not yield exponential mixing. 
However, an argument of Pollicott \cite{Pollicott_mixing} shows that 
the exponential mixing of the BMS measure $m^{\BMS}$ can be deduced from spectral bounds on the  transfer operators associated to a Markov section for the geodesic flow, via the Laplace transform of the correlation function.
Dolgopyat introduced a powerful method for obtaining such bounds for the geodesic flow on a compact negatively curved manifold, using a quantitative non-integrability property of the stable and unstable foliations of the geodesic flow; this approach is now commonly known as Dolgopyat's methods \cite{Dolgopyat}.

For convex cocompact $\Ga$, the geodesic flow on $\T^1(\cM)$ admits a finite Markov section. Building on Dolgopyat's ideas, Sarkar and Winter proved exponential mixing for the frame flow with respect to $m^{\BMS}$  with no restriction on $\delta$ \cite{SW} (see also \cite{Stoyanov} for an earlier work on the geodesic flow). 

For geometrically finite groups with cusps, the situation is more delicate since no finite Markov section exists. In a major advance, Li and Pan \cite{LP} constructed a countable Markov section for geodesic flow in this setting and established an exponential mixing by extending Dolgopyat's method to suspension flows over countable shifts.
In joint work with Sarkar, they extended this result to the frame flow \cite{LPS}. Independently, Khalil \cite{Khalil} obtained a similar mixing property for the geodesic flow using a different approach based on anisotropic Banach spaces.
Combined with an effective version of
Roblin's transversal intersection argument (\cite{Oh_Winter_JAMS}), these results yield the exponential local mixing for the Haar measure:

\begin{theorem} [\cite{SW}, \cite{LPS}] \label{lps} Let $\Ga<G$ be a geometrically finite Zariski dense subgroup.
Then there exists $\eta>0$ such that
for any 
$f_1, f_2$ on $C_c (\Ga\ba G)$ with $\|f_1\|_{C^1}, \| f_2\|_{C^1} <\infty$,
    \be \lim_{t\to +\infty} e^{(d-1-\delta)t } \left\langle a_t. f_1, f_2\right\rangle  =\frac{1}{|m^{\BMS}|} {m^{\op{BR_*}}(f_1) \cdot m^{\op{BR}} (f_2) } + O \left( e^{-\eta  t }\|f_1\|_{C^1} \| f_2\|_{C^1} \right) \ee
    where $\|f_i\|_{C^1}$ denotes the $C^1$-norm of $f_i$.
\end{theorem}
Unlike Theorem \ref{eo}, the error exponent $\eta$ here is not explicit.

\subsection{Uniform exponential mixing and application to the affine sieve}
In many arithmetic settings, one needs exponential mixing estimates that hold
 uniformly over families of congruence covers. Establishing such uniformity not only strengthens the dynamical picture but also yields important applications in number theory, most notably through the affine sieve.

Selberg's $\frac{3}{16}$-theorem about congruence covers for $\SL_2(\z)\ba \bH^2$ can be formulated in terms of uniform strong mixing as follows: For $q\in \mathbb N$, let
$\Ga_q=\{\ga\in \SL_2(\z): \ga =e \;\; (\text{mod q})\}$ be the congruence subgroup of level $q$. Then for any $\e>0$,
we have that for any 
$f_1, f_2\in C_c (\Ga_q\ba \SL_2(\br))$,
    \be \lim_{t\to +\infty} \langle a_t. f_1, f_2 \rangle_{L^2(\Ga_q\ba \SL_2(\br))}  =
    \frac{1}{|m^{\Haar}|} m^{\Haar} (f_1 ) \cdot m^{\Haar} (f_2)  + O \left( e^{(-1/4+\e)  t} \mathcal{S}^2(f_1) \mathcal{S}^2(f_2)  \right) \ee
   where $m^{\Haar}$ denotes the $\SL_2(\br)$-invariant measure on $\Ga_q\ba \SL_2(\br)$ induced from $dx$. The crucial point is that the exponent $\frac 14$ is uniform over all $q$ and that $\frac{3}{16}= \frac{1}4(1-\frac 14)$. The eigenvalue conjecture predicts that $\frac14$ should be replaced by $\frac 12$.
    
We describe an analogue of Selberg's $\frac{3}{16}$-theorem for thin subgroups of $G=\SO(d,1)^\circ$. Let $\Ga<G$ be a {\it thin} group, i.e., $\Gamma$ is a Zariski dense subgroup of an arithmetic subgroup $G(\mathbb Z)$ of infinite index. 
For each $q\in \mathbb N$, write $\Gamma_q<\Ga$ for its level $q$ congruence subgroup.
The expander machinery of Bourgain-Gamburd-Sarnak (\cite{BGS_inv}, \cite{BGS_acta}) and its generalization by Golsefidy-Varju \cite{GolsefidyVarju} show that  if $\delta >\frac{d-1}{2}$,
there exists a finite set of primes $S$ such that the family $\mathcal F:=\{\Gamma_q: \text{$q$ is square-free with no factors in $S$}\}$ has a uniform spectral gap: \be\label{ef} \eta_{\cal F}:=\inf_{\Gamma_q\in \mathcal F} \{ \delta -s_1(q), 1\} >0,\ee 
where $s_1(q) (d-1-s_1(q))$ is the second smallest eigenvalue of $\Delta$ on $L^2(\Gamma_q\ba \bH^{d})$. Together with this, Theorem \ref{eo} implies the following:
\begin{theorem} [\cite{EO_duke}]  If $\delta >\frac{d-1}{2}$, the family of geodesic flows over $\Gamma_q\in \mathcal F$ is uniformly 
exponentially mixing: 
 for any $\e>0$ and any $f_1, f_2\in C_c(\Ga_q\ba G)^M$,
  $$ \lim_{t\to +\infty} e^{(d-1-\delta)t } \langle a_t. f_1, f_2 \rangle_{L^2(\Ga_q\ba G)}  =\frac{1}{|m^{\BMS}|} m^{\op{BR}}(f_1) \cdot m^{\op{BR_{*}}}(f_2) + O_\e \left( e^{(-\eta_{\cal F} +\e )t }\mathcal S^m(f_1)\mathcal S^m(f_2) \right) $$
    where  $m^{\BMS}$, $m^{\BR}$ and $m^{\BR_*}$  on $\Ga_q\ba G$ are the measures induced from the corresponding fixed ones on $\Ga\ba G$.
\end{theorem}

As before, when $\delta\le \frac{d-1}{2}$, the $L^2$-spectral approach no longer applies. For the surface case ($d=2$), Winter and the author
established uniform exponential mixing by combining the Bourgain-Gamburd-Sarnak expander machinery and Dolgopyat's method \cite{Oh_Winter_JAMS}. Sarkar extended these arguments
to higher dimensions $d\ge 3$, using the results in \cite{SW} and \cite{LPS}.  As in the general exponential mixing theorem Theorem \ref{lps},  the  error exponent $\eta$ in this setting is not explicit.
\begin{theorem} [\cite{Oh_Winter_JAMS}, \cite{Sarkar_uniform_cc}, \cite{Sarkar_uniform_gf}] The family of frame flows over $\Gamma_q\in \mathcal F$ is uniformly 
exponentially locally mixing: there exists $\eta>0$ such that
for any $f_1, f_2\in C_c(\Ga_q\ba G)$, 
  $$ \lim_{t\to +\infty} e^{(d-1-\delta)t } \langle a_t. f_1, f_2 \rangle_{L^2(\Ga_q\ba G)}  =\frac{1}{|m^{\BMS}|} m^{\op{BR}}(f_1) \cdot m^{\op{BR_{*}}}(f_2) + O_\e \left( e^{-\eta t }\mathcal \|f_1\|_{C^1}\mathcal \|f_2\|_{C^1} \right) .$$
\end{theorem}

These uniform exponential mixing results form a key dynamical input for counting orbital points with almost prime coordinates with respect to the archimedean norm (cf. \cite{KO}, \cite{MohammadiOh_JEMS}), one of the central themes of the affine sieve developed by Bourgain–Gamburd–Sarnak.

\begin{corollary} Let $Q(x_1, \cdots, x_{d+1})$ be an integral quadratic form of signature $(d,1)$ with $d\ge 2$. Let $\G$ be a geometrically finite and Zariski dense subgroup of $\SO_Q({\mathbb Z})$.  Let  $w_0\in \mathbb Z^{d+1}$; when $d=2$, assume that $Q(w_0)\le 0$. For any  norm $\|\cdot \|$ on $\br^{d+1}$ and $1\le r\le d+1$, we  have the following for all $T\ge 1$:
\begin{enumerate}
 \item $\#\{{\bf x}\in w_0\G:  \|{\bf x}\|<T, \;\; \text{$x_j$ is prime for all $j=1, \cdots, r$}\}\ll
\frac{T^\delta}{(\log T) ^r} ;$
\item For some $R>1$, $\#\{{\bf x}\in w_0\G :  \|{\bf x}\|<T, \; x_1\cdots x_r \text{ has at most $R$-prime factors}\}\gg
\frac{T^\delta}{(\log T) ^r} .$
\end{enumerate}
\end{corollary}

This result generalizes a theorem with Kontorovich from \cite{KO} that for any bounded integral\footnote{this means that curvatures of all circles in $\cal P$ are integers.} Apollonian circle packing $\cal P$, the number of circles of prime curvature (=reciprocal of radius) at most $T$
satisfies
$$\#\{C\in \cal P: \text{curv}(C) \text{ prime at most $T$}\} \ll \frac{T^{\delta_{\mathsf A}}}{\log T} \quad\text{ for all $T\ge 1$} $$
(see Figure \ref{ap}). Here,  quadruples of curvatures of four mutually tangent circles correspond to points in an $\Ga_{\cal P}$-orbit on the zero locus of the Descartes quadratic form.

In summary, uniform exponential mixing provides the dynamical engine of the affine sieve, forging a bridge  between dynamics on hyperbolic manifolds and distribution of almost primes  in orbits.

\section{Matrix coefficients and Mixing in higher-rank infinite volume}\label{s:lmhigher} 
We now turn to the higher-rank setting, where the main difference from rank one is that the required renormalization depends on the direction of the diagonal flow.

\subsection{Local mixing for self-joining quotients} Let $$G=G_1\times \cdots \times G_n, \quad  G_i=\SO(d,1)^\circ.$$
The rank of $G$ is $n$; hence $n\ge 2$ corresponds to higher rank.
Let $\Ga<\SO(d,1)^\circ$ be a Zariski dense discrete group, and let $\rho_1, \rho_2, \cdots, \rho_n$ be discrete faithful representations of $\Ga$ into $\SO(d,1)^\circ$, where $\rho_1$
is the inclusion map.
Consider the self-joining of $\Gamma$ via $\rho=\prod_{i=1}^d \rho_i$:
$$\Gamma_\rho 
=(\prod_{i=1}^n \rho_i  )(\Ga)=\{\big(\rho_1(\ga), \cdots , \rho_n(\ga)\big): \ga\in \Gamma\}<G.$$
We study the dynamics of diagonal flows on the quotient $\Ga_\rho\ba G$.
The maximal connected diagonalizable subgroup $A<G$ is of the form
$$A=\{ a_{\u}= (a_{u_1}, \cdots, a_{u_n} ) :\mathsf{u}=(u_1, \cdots, u_n) \in \br^n\} $$
where $\{a_t:t\in \br \}$ is a one-parameter diagonalizable subgroup of $\SO(d,1)^\circ$. Let $\fa\simeq \br^n$ be the Lie algebra of $A$ with positive Weyl chamber $\fa^+=\{\u=(u_1, \cdots, u_n): u_i\ge 0 \text{ for all $i$}\}$. For each non-zero $\u\in \fa^+$,
define the one-parameter subgroup $$A_{\u}=\{a_{t\u}: t\in \br\}.$$

\begin{figure} [htbp]\label{limitcone}\begin{center}
  \includegraphics [height=3cm]{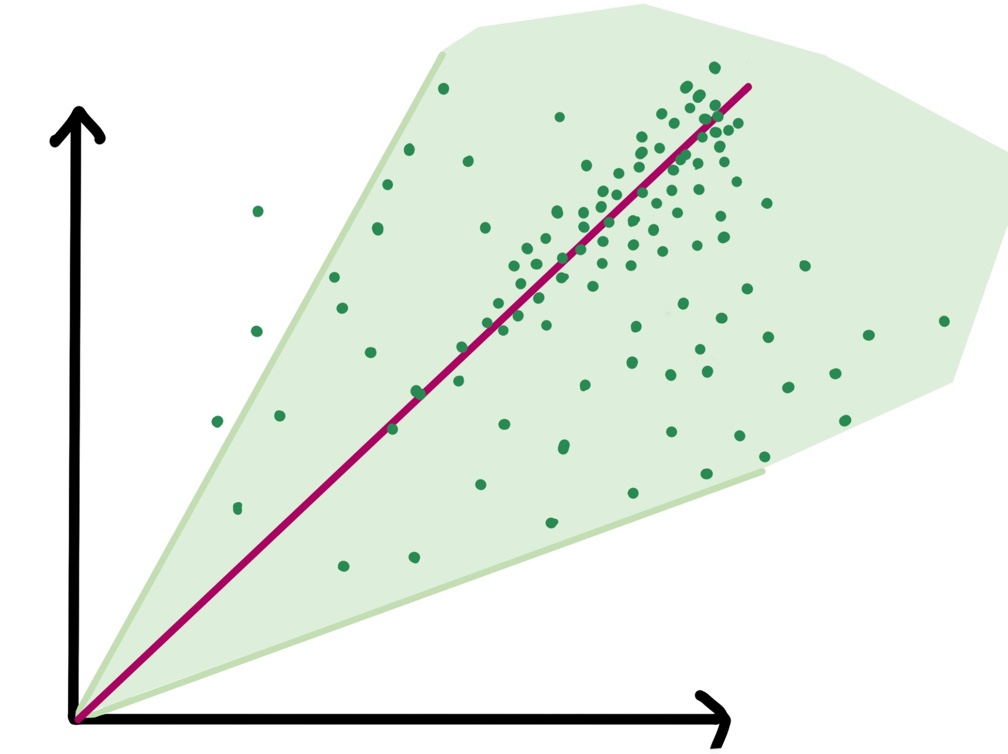}
 \end{center}\caption{Limit cone and a ray in its interior for $n=2$}
\end{figure}  The dynamical behavior of $A_{\u}$-action on $\Ga_\rho\ba G$ depends strongly on the direction $\u$. The set of directions where non-trivial dynamics occur forms the limit cone of $\Ga_\rho$, introduced by Benoist:
$$ \cal L_\rho:= \{ \lim_{k\to \infty} t_k \u_k\in \fa^+: t_k\to 0, \u_k\in \mu(\Ga_\rho)\} $$
where $\mu(\Ga_\rho)=\{ \mu(\ga_\rho)=\left( d(\rho_1(\ga)o, o), \cdots, d(\rho_n(\ga) o, o)  \right)\in \fa^+: \ga\in \Ga\}$ is the Cartan projection of $\Ga_\rho$.  Figure~\ref{limitcone} illustrates the limit cone and an interior ray.
 Benoist proved that
if $\Ga_\rho$ is Zariski dense, its limit cone $\cal L_\rho$ is a convex cone with non-empty interior \cite{Benoist1997proprietes}. 

Another essential tool is the growth indicator $\psi_\rho:\fa^+\to \br$ of $\Ga_\rho$, introduced by Quint \cite{Quint_indicator}.
For $\u\in \fa^+$,
$$\psi_{\rho}(\u):=\|\u\| \inf_{\u \in\mathsf C} \tau_{\mathsf C}$$
where $\mathsf C$ ranges over all open cones containing $\u$ and $\tau_{\mathsf C}$ is the abscissa of convergence of the series $\sum_{\ga\in\Ga,\,\mu(\ga_\rho)\in\mathsf C}e^{-t\|{\mu(\ga_\rho)}\|}.$  Quint proved that
$\psi_\rho$ is upper semi-continuous and concave, and $\L_\rho=\{\psi_\rho \ge 0\}$.  In higher rank, it records how the growth rate of $\Ga_\rho$-orbit in $\prod_{i=1}^n \bH^d$ varies with the chosen direction.

For the rest of this section, we suppose that 
$$\text{each $\rho_i$ is convex cocompact and that no two $\rho_i$ are conjugate in $\op{Isom}(\bH^d)$.}$$
In joint work with Edwards and Lee, we established
the following local mixing theorem for directional flows.  
\begin{theorem}[\cite{ELO_anosov}] \label{elo_mix} 
 For any $\u\in \op{int}\L_\rho$, there exists $\kappa_{\u}>0$ such that for all $f_1, f_2\in C_c(\Gamma\ba G)$,
$$\lim_{t\to +\infty} t^{(n -1)/2} e^{t(\sigma-\psi_\rho)(\u)} \langle a_{t\u}.f_1, f_2\rangle 
= \kappa_{\u} \cdot m^{\BR}_{\u} (f_1) \,m^{\BR_*}_{\u} (f_2) $$
where $\sigma(\u)=u_1+\cdots +u_n$ is the sum of all positive roots of $(\frak g, \frak a)$ and
$m_{\u}^{\BR}$ and  $m^{\BR_*}_{\u} $ are the stable and unstable Burger-Roblin measures on $\Ga\ba G$ associated to $\u$.
\end{theorem}

For counting results such as Theorem \ref{elo}, one needs 
not only asymptotics along a fixed direction but also uniform control of  across all directions in $ \fa^+$.
A linear form $\psi\in\fa^*$ is said to be tangent to the growth indicator $\psi_\rho$ at $\u\in \fa^+$ if $\psi\ge \psi_\rho$ and $\psi(\u)=\psi_\rho(\u)$. 
In the setting of Theorem \ref{elo_mix}, the growth indicator $\psi_\rho$ is strictly concave and analytic in $\inte \L_\rho$ (\cite{Quint_indicator}, \cite{samb_hyper}, \cite{PS_eigenvalues}). As a consequence, there exists a unique tangent form $\psi_{\u}$ at $\u$.  The positive Weyl chamber
$\fa^+$ admits the parametrization
$$\br_{\ge 0}\times \ker \psi_{\u}\to \fa, \quad (t, {\mathsf w})\mapsto t\u+\sqrt t {\mathsf w} .$$

\begin{theorem} [\cite{ELO_anosov}, \cite{ELO_unique}]\label{elo_full}
Let $\u\in \op{int}\L_\rho$. For any $f_1, f_2\in C_c(\Ga\ba G)$ and ${\mathsf w}\in \ker \psi_{\u}$,
$$ \lim_{t\to +\infty} t^{(n-1)/2} e^{(\sigma -\psi_\rho) (t\u +\sqrt t\w)}  \langle a_{t\u+\sqrt t\w }.f_1, f_2\rangle = \kappa_{\u}\,e^{-I(\w)} m_{\u}^{\BR} (f_1) \, m_{\u}^{\BR_*} (f_2)
$$ where $I(\w)=\tfrac{\|\w\|_*^2-\langle \w, \u \rangle_*^2}{\|\u\|_*^2} $  for some inner product $\langle\cdot,\cdot\rangle_*$ on $\frak a$.
Moreover, the left-hand side is uniformly bounded by a fixed constant multiple of $e^{- I(\w)}$ for all sufficiently large $t$.
\end{theorem}

\subsection{Burger-Roblin measures in higher rank} 
We explain the Burger-Roblin measures $m_{\u}^{\BR}$
and $m_{\u}^{\BR_*}$, as introduced in \cite{ELO_anosov}.
The Furstenberg boundary of $G$ is $\cal F=\prod {\bS}^{d-1}$, and
the limit set
of $\Ga_\rho$ is
$\La_\rho=\{ \xi_\rho=(f_1(\xi), \cdots, f_n(\xi)) : \xi \in \La\} $
where $f_i$ denotes the $\rho_i$-boundary map. Fix $o\in \prod \bH^d$ and let $K=\op{Stab}_G(o)$.
For any linear form $\psi\in \fa^*$ tangent to $\psi_\rho$ at some $\u\in \inte \fa^+$, Quint constructed a higher rank analogue of Patterson-Sullivan measure, a measure supported on $\La_\rho$ satisfying 
    \be\label{higher} \frac{d\gamma_* \nu}{d\nu}(\xi_\rho)= e^{\psi(\beta_{\xi_\rho}(o,\ga_\rho o)) }  \quad\text{ for all $\ga_\rho\in \Ga_\rho$},\ee  where $\beta$ is the $\fa$-valued Busemann map \cite{Quint2002Mesures}; it is called a $(\Ga_\rho, \psi)$-Patterson-Sullivan measure.
Under the hypothesis of Theorem \ref{elo_mix},
there exists a unique $(\Ga_\rho, \psi_{\u})$-Patterson-Sullivan measure, denoted $\nu_{\u}$, and for different directions, these measures are mutually singular \cite{LO_invariant}. 

 Let $M$ be the centralizer of $A$ in $K$, $m_o$ the $K$-invariant probability measure on $\cal F$, and $ds$ the Lebesgue measure on $\fa$.  
The Burger-Roblin measure $m^{\BR}_{\u}$ is induced from the following  $\Ga$-invariant measure on $G/M$:
$$
d \tilde m^{\BR}_{\u}(g) =
e^{\psi_{\u} (\beta_{g^+}(o, g(o)))}\; e^{\sigma (\beta_{g^-}(o,go)) }\;d\nu_{\u}(g^+) dm_o(g^-) ds ;$$
where $g^{\pm}=(g_1^{\pm}, \cdots, g_n^{\pm}) $ and $s=\beta_{g^+}(o, go)\in \fa$.
This measure $m^{\BR}_{\u}$ is an $N$-invariant ergodic measure where $N$ is the stable horospherical subgroup for $A_{\u}$ (\cite{LO_invariant}, \cite{LO_ergodic_decom}). Unlike the rank one convex cocompact case, where the BR-measure is the {\emph{unique}} $N$-ergodic measure on $\Ga\ba \SO (d,1)^\circ$ not supported on a closed $NM$-orbit, in higher rank, one obtains a continuous family of $N$-ergodic measures on $\Ga_\rho\ba G$ parametrized by directions in $\op{int} \L_\rho$. They yield all $N$-ergodic measures supported on the set of {\it directionally recurrent} points \cite{LLLO}.
It remains an open question whether these measures exhaust all possible $N$-ergodic measures in $\Ga_\rho\ba G$, except for those supported on closed $NM$-orbits.
Similarly, the unstable Burger–Roblin measure $m^{\BR_*}_{\u}$ is defined by switching the roles of $g^+$ and $g^-$ in the above formula.

\subsection{Local mixing of Bowen-Margulis-Sullivan measures in higher rank}
The Bowen-Margulis-Sullivan measure $m^{\BMS}_{\u}$ is the $AM$-invariant measure on $\Ga_\rho \ba G$ induced from
\begin{align*}
d \tilde m^{\BMS}_{\u}(g)& =
e^{\psi_{\u} (\beta_{g^+}(o, go))}\; e^{\psi_{\u}( \beta_{g^-}(o,go)) }\;d\nu_{\u} (g^+) d\nu_{\u}(g^-) ds.
\end{align*}
A crucial difference between rank one and higher rank is that 
in higher rank $m^{\BMS}_{\u}$ is an {\it infinite} measure. Hence strong mixing does not make sense. Nonetheless,  $m^{\BMS}_{\u}$ is $A$-ergodic (\cite{LO_ergodic_decom}) and 
the following local mixing result was proved by Sambarino for the $M$-invariant functions and by Chow-Sarkar in general:
\begin{theorem}[\cite{sambarino_orbit}, \cite{CS_local}] \label{sam}
 For any $\u\in \op{int}\L_\rho$, there exists $\kappa_{\u}>0$ such that for all $f_1, f_2\in C_c(\Gamma_\rho\ba G)$,
$$\lim_{t\to +\infty} t^{(n -1)/2}  \int_{\Gamma_\rho\ba G} f_1(x a_{t\u})  f_2(x) \,dm_{\u}^{\BMS} (x) 
= \kappa_{\u} \cdot m_{\u}^{\BMS} (f_1) \, m_{\u}^{\BMS} (f_2). $$
\end{theorem}
Theorem \ref{elo_mix} is in fact deduced from this result via a higher rank version of transversal intersection argument.

 If $\Ga$ is a normal subgroup of a convex cocompact Zariski dense subgroup $\Ga_0<\SO(d,1)^\circ$ with $\Ga/\Ga_0\simeq {\mathbb Z}^{n-1}$, then in joint work with Pan \cite{Oh_Pan_ZD}, we proved that for all $f_1, f_2\in C_c(\Gamma\ba \SO(d,1)^\circ)$,
$$\lim_{t\to +\infty} t^{(n -1)/2}  \int_{\Gamma\ba \SO(d,1)^\circ} f_1(x a_{t})  f_2(x) \,dm^{\BMS} (x) 
= \kappa \cdot m^{\BMS} (f_1) \, m^{\BMS} (f_2) $$
for some constant $\kappa>0$. The factor $t^{(n-1)/2}$ arises because the support of $m^{\BMS}$ is the ${\mathbb Z}^{n-1}$-cover
of a compact flow space, namely the renormalized frame bundle of $\Ga_0\ba \bH^d$.

Likewise, the same factor in Theorem \ref{sam} reflects that the support of $m^{\BMS}_{\u}$ is an $\br^{n-1}$-bundle over a compact flow space.
A key structural property of $m_{\u}^{\BMS}$ for Anosov subgroups is that
it decomposes as the product of a rank one BMS-type measure on a compact space and the Lebesgue measure on the kernel of the linear form $\psi_{\u}$. Letting $\La_\rho^{(2)}=(\La_\rho\times \La_\rho - \text{diag} )$,
we have that the $\Ga_\rho$-action on $\La_\rho^{(2)} \times \br$ given by
$\ga. (\xi,\eta, t)=(\ga \xi, \ga \eta, t+\psi_{\u}(\beta_\xi(e, \ga))$ is properly discontinuous and cocompact.  Hence $\Omega_{\u}:=  \Ga\ba ( \La_\rho^{(2)}\times \br)$ is a compact metrizable space. The projection
$$ \Ga_\rho \ba (\La_\rho^{(2)} \times \fa) \to \Omega_{\u} \quad\quad  [(\xi, \eta, \mathsf{v})]\mapsto [(\xi, \eta, \psi_{\u} (\mathsf{v}))]$$ defines a $\ker \psi_{\u}$-vector bundle over 
 $\Omega_{\u}$ \cite{sambarino_orbit}. It follows from the work of Bridgeman-Canary-Labourie-Sambarino \cite{BCLS_gafa} that the translation flow on $\Omega_{\u}$ is a metric Anosov flow, which is H\"older conjugate to a reparameterization of the Gromov geodesic flow associated to $\Gamma$. This structure provides a Markov section on $\Omega_{\u}$ and hence a finite symbolic coding of the flow.
 Theorem \ref{sam} then follows from the control of the $\fa$-valued transfer operator on this symbolic model by applying the Fourier inversion theorem to the correlation function.

\subsection{Drunken bird vs. drunken person} 
The vector-bundle structure yields a homeomorphism
$$\Ga_\rho \ba ( \La_\rho^{(2)}\times \fa ) \simeq \Omega_{\u} \times \br^{n-1} $$
which suggests an analogy with random walks on Euclidean spaces.
 In dimensions one and two, a random walk almost surely returns home (a "drunken person"), while in higher dimensions, a "drunken bird" tends to fly off to infinity.
This leads to the following rank dichotomy, obtained
 jointly with Burger, Landesberg and Lee.
 \begin{theorem}[\cite{BLLO}]\label{bllo}
 Let $\mathsf u\in \inte \L_\rho$.
The $A_{\u}$-action on $(\Ga_\rho\ba G, m^{\BMS}_{\u})$ is ergodic if and only if $n \le 3$.
 \end{theorem}

For example, if $\Ga<G=\SO(d,1)^\circ\times \SO(d,1)^\circ $ is a Zariski dense
self-joining of a convex cocompact subgroup of $\SO(d,1)^\circ$, then the flow
$A_{\u}$ on $(\Ga\ba G, m^{\BMS}_{\u})$ 
is ergodic for any $\u\in \inte \L_\rho$. This ergodicity was one of the main ingredients in the proof of Theorem \ref{ko2b}.

Theorem \ref{bllo} follows from a general Hopf-Tsuji-Sullivan type criterion for the ergodicity of the action of a one parameter diagonal subgroup $A_{\u}$ on $\Ga\ba G$, where $G$ is a connected semisimple real algebraic group and $\Ga<G$ a Zariski dense discrete subgroup \cite{BLLO}. The ergodicity criterion relies  on the measure
being $\u$-balanced, that is, $\limsup_{t\to +\infty}\frac{\int_0^t m_{\u}^{\BMS}(\cal O_1\cap \cal O_1 a_{t\u}) dt }{\int_0^t m_{\u}^{\BMS}(\cal O_2\cap \cal O_2 a_{t\u})dt} <\infty$ for any bounded open subsets $\cal O_1, \cal O_2$. This condition is verified for $\Ga_\rho$ by Theorem \ref{sam}.

\subsection{Borel Anosov subgroups}\label{s:anosov}  Let $G=KA^+K$ be a Cartan decomposition. Recall that a finitely generated subgroup $\Ga<G$ is called {\emph{Borel Anosov}} if there exists $c>0$ such that for all $\ga\in \Ga$,	$$ \alpha (\mu(\ga)) \ge  c|\ga| -c^{-1}$$ for all simple roots $\alpha$ of $(\frak g, \fa)$ where $|\gamma |$ denotes the word length of $\ga$ with respect to a fixed finite set of generators of $\Ga$ and $\mu(\ga)\in \fa^+$ is the Cartan projection of $\ga$, i.e., $\ga\in K \exp(\mu(\ga)) K$.  This condition is stronger than the quasi-isometric embedding property for the orbit map $\ga\mapsto \ga o\in G/K$
    where $K=\op{Stab}(o)$: it requires a quasi-isometric type control for {\it every} simple root $\alpha$. The class of Borel Anosov subgroups forms a distinguished subclass of the general family of Anosov subgroups, since in general the Anosov property may be imposed relative
to any subset of the simple roots.  Only in the Borel Anosov case does the associated dynamical system coincide with the homogeneous space $\Ga\ba G$; for other types of Anosov subgroups, the relevant dynamical space is no longer $G$-homogeneous.  The notion of Anosov subgroups was introduced by Labourie for surface groups~\cite{Labourie2006anosov}, extended by  Guichard-Wienhard for Gromov hyperbolic groups~\cite{Guichard2012anosov}, and formulated in the general language used here by Kapovich-Leeb-Porti \cite{KLP_Anosov}.
  
 Analogues of Theorems \ref{elo_full} and \ref{sam} hold for all Zariski dense Borel Anosov subgroups of a connected semisimple real algebraic group $G$. In our setting, a self-joining $\Ga_\rho$ is Borel Anosov if and only if all $\rho_i$ are convex cocompact representations. This explains the hypothesis in Theorem \ref{elo} and in the other theorems of this section.
A natural question is whether these results  extend to self-joinings of geometrically finite representations. More generally, one may ask whether an analogous local mixing phenomenon holds for Zariski-dense {\it relatively Borel Anosov} subgroups—viewed as higher-rank analogues of geometrically finite groups—just as Anosov subgroups generalize convex cocompact ones in higher rank.

\subsection{Rank-one vs. higher-rank contrast for $L^2$-spectrum}
In rank one, as noted in Theorem \ref{lax}, the existence of an $L^2$-eigenfunction at the bottom of the Laplace spectrum depends on the critical exponent $\delta$. By contrast, for a general semisimple real algebraic group $G$ and any Zariski dense discrete subgroup $\Ga<G$, the bottom of the $L^2$-spectrum on $\Ga\ba X$ (where $X$ is the associated symmetric space) is never an atom, equivalently, there is no positive $L^2$-eigenfunction of $\Delta$. The only exception occurs in the product case: up to commensurability, $\Ga=\Ga_1\times \Ga_2$ where 
$\Ga_1$ is a discrete subgroup of a rank one factor $G_1$ and $\Ga_2$ is a lattice in $G_2$ with $G=G_1 G_2$. In particular, if $G$ is simple and of higher rank, then the bottom of $L^2(\Ga\ba X)$ is never an atom. This was shown in joint work with Edwards, Fraczyk, and Lee \cite{EFLO}. These results naturally lead to the question of how the continuous part of the spectrum behaves and, in particular, when the representation $L^2(\Gamma \backslash G)$ is tempered, meaning that all $K$-finite matrix coefficients are $L^{2+\e}$-integrable for any $\e>0$.

Temperedness of $L^2(\Ga\ba G)$ is closely related to the orbital growth of $\Ga$. For Borel Anosov subgroups, it was shown in joint work with Edwards \cite{Edwards_Oh_tem} that  $L^2(\Ga\ba G)$ is tempered (in which case we call $\Ga$ tempered) if and only if the growth indicator of $\Ga$ is bounded above by the half-sum of all positive roots.  More recently, Lutsko, Weich, and Wolf \cite{LWW} extended this growth-indicator criterion to all discrete subgroups and proved that $L^2(\Ga\ba G)$ is tempered for any Borel Anosov subgroup (possibly except when the reduced root system of $G$ is $A_2$), thereby confirming the conjecture in \cite{KMO_tent}. Finally, in joint work with Fraczyk \cite{FO}, we constructed Zariski-dense, non-tempered subgroups in higher rank. In contrast to the rank-one setting, where non-tempered subgroups always give rise to an atom at the base of the $L^2$-spectrum, these examples show that in higher rank, non-temperedness can occur even without such an atom—a phenomenon that appears to be genuinely higher-rank in nature.

\end{document}